\documentclass[twoside]{amsart}

\usepackage{amsmath,amssymb,amsfonts,amsthm,latexsym}
\usepackage{amscd,graphicx,color,enumerate}
\usepackage{hyperref}
\usepackage[all]{xy}

\numberwithin{equation}{section}

\newtheorem{theorem}{Theorem}[section]
\newtheorem{lemma}[theorem]{Lemma}
\newtheorem{proposition}[theorem]{Proposition}
\newtheorem{corollary}[theorem]{Corollary}

\newtheorem{mainthm}{Theorem}
\newtheorem{mainprop}[mainthm]{Proposition}
\newtheorem{maincor}[mainthm]{Corollary}

\theoremstyle{definition}
\newtheorem{definition}[theorem]{Definition}

\newtheorem{example}[theorem]{Example}

\theoremstyle{remark}
\newtheorem{remark}[theorem]{Remark}
\newtheorem{remarks}[theorem]{Remarks}

\renewcommand{\P}{\mathbb{P}}
\renewcommand{\O}{\mathcal{O}}
\newcommand{\C}{\mathbb{C}}
\newcommand{\Sp}{\mathbb{S}}
\newcommand{\Q}{\mathbb{Q}}
\newcommand{\R}{\mathbb{R}}

\newcommand{\Z}{\mathbb{Z}}
\newcommand{\N}{\mathbb{N}}

\newcommand{\SL}{\operatorname{SL}}

\newcommand{\GL}{\operatorname{GL}}
\newcommand{\id}{\operatorname{id}}

\newcommand{\Aut}{\operatorname{Aut}}

\newcommand{\Ker}{\operatorname{Ker}}

\newcommand{\rank}{\operatorname{rank}}

\newcommand{\git}{/\!\!/}
\newcommand{\tosim}{\xrightarrow{\sim}}

\newcommand{\pr}{\mathrm{pr}}
\newcommand{\sing}{\mathrm{sing}}
\newcommand{\sm}{\mathrm{sm}}
\newcommand{\codim}{\operatorname{codim}}

\newcommand{\inv}{^{-1}}

\newcommand{\NN}{\mathcal N}

\def\lie#1{{\mathfrak #1}}
\newcommand{\lieg} {\lie g}
\newcommand{\liek} {\lie k}

\renewcommand{\phi}{\varphi}
\newcommand{\A} {\lie A}
\begin{document}

\title{Isomorphisms of Symplectic Torus Quotients}
\author[H.-C. Herbig]{Hans-Christian Herbig}
\email{herbighc@gmail.com}
\address{Departamento de Matem\'{a}tica Aplicada, Universidade Federal do Rio de Janeiro,
Av. Athos da Silveira Ramos 149, Centro de Tecnologia - Bloco C, CEP: 21941-909 - Rio de Janeiro, Brazil}

\author[G. W. Schwarz]{Gerald W. Schwarz}
\email{schwarz@brandeis.edu}
\address{Department of Mathematics, Brandeis University,
Waltham, MA 02454-9110, USA}

\author[C. Seaton]{Christopher Seaton}
\email{seatonc@rhodes.edu}
\address{Department of Mathematics and Computer Science,
Rhodes College, 2000 N. Parkway, Memphis, TN 38112, USA}


\begin{abstract}
We call a reductive complex group $G$ \emph{quasi-toral\/} if $G^0$ is a torus. Let $G$ be quasi-toral and let $V$ be a faithful $1$-modular $G$-module.
Let $N$ (the shell) be the zero fiber of the canonical moment mapping $\mu\colon V\oplus V^*\to\lieg^*$. Then $N$ is a complete intersection  variety with rational singularities. Let $M$ denote the categorical quotient $N\git G$.     We show that $M$ determines $V\oplus V^*$ and $G$, up to isomorphism, if $\codim_N N_\sing\geq 4$. If $\codim_NN_\sing=3$, the lowest possible, then there is a  process to produce an algebraic (hence quasi-toral)  subgroup $G'\subset G$ and a  faithful $1$-modular $G'$-submodule
 $V'\subset V$ with shell $N'$  such that   $\codim_{N'}(N')_\sing\geq 4$. Moreover, there is a $G'$-equivariant morphism $N'\to N$ inducing an isomorphism $N'\git G'\tosim N\git G$. Thus, up to isomorphism,  $M$ determines $V'\oplus (V')^*$   and $G'$, hence also $N'$.
We establish similar results for real shells and real symplectic quotients associated to  unitary modules   for compact Lie groups.
\end{abstract}

\maketitle

\section{Introduction}\label{sec:introduction}

Let $G$ be a reductive complex group and $V$ a $G$-module. Let $U=V\oplus V^*$ and let $\mu\colon U\to\lieg^*$ be the canonical
homogeneous moment mapping where $\lieg$ is the Lie algebra of $G$. Let $N_V$ or $N_U$ (or just $N$) denote $\mu\inv(0)$, which we call the \emph{shell}. It is a complete intersection variety if and only if $V$ is $1$-modular (see \S \ref{sec:background}). The ``general'' $G$-module $V$ is $1$-modular
if $G$ is a torus or semisimple,
see Remark \ref{rem:general} below.
Now $U$ has a canonical  $G$-invariant symplectic structure. If $V'$ is a  Lagrangian $G$-submodule  of $U$, then $U\simeq V'\oplus (V')^*$ and we get the same moment mapping and shell, up to isomorphism \cite[Lemma 2.1]{HerbigSchwarzSeaton3}. We say that $V$ satisfies a property (P)  \emph{up to change of Lagrangian subspace, abbreviated  UTCLS,\/} if there is a Lagrangian $G$-submodule $V'$ of $U$ satisfying (P).   We say that $V$ is \emph{stable\/} if it has
a nonempty open subset of closed $G$-orbits.
Let $M=N\git G$ denote the categorical quotient,
called the \emph{complex symplectic quotient associated to $V$}.
The symplectic form on $U$ induces a structure of Poisson algebra on $\C[N]$ and on $\C[N]^G\simeq \C[M]$.

 There has been much investigation of whether or not there are various kinds of isomorphisms $M\simeq M'$ where $M'$ corresponds to a $G'$-module $V'$ for $G'$   a reductive complex group of dimension less that $\dim G$. For example, when is $M$ an orbifold, i.e., when is $M\simeq M'$ where $G'$ is finite? Cases where $M$ is an orbifold have been noted often in the literature and the question of when this does or does not occur has been  systematically treated in \cite{FarHerSea} for tori, in \cite{HerbigSeaton2} for $G=\C^\times$ and  more generally in \cite{HerbigSchwarzSeaton}.
The case where $M'$ is a $\C^\times$ quotient was treated in \cite{HerbigLawlerSeaton20}.
Previous results were unable to answer the orbifold question for $G$ a torus and $V$ a  faithful $G$-module which  is $1$-large
but not $2$-large (see \S \ref{sec:background}).
We say that $G$ is \emph{quasi-toral} if $G^0$ is a torus.
Our results here give necessary and sufficient conditions for algebraic
isomorphisms $M\simeq M'$ when $G$ and $G'$ are quasi-toral and the modules $V$ and $V'$ are  faithful and $1$-modular.

We show the following (Theorem  \ref{thm:i.groups}).
\begin{mainthm}\label{mainthm:1}
Let $G$ be a torus and $V$  a $1$-modular faithful $G$-module.
Let $H$ be the isotropy group of a nonzero closed orbit in  $N$. Then, UTCLS,  $V$ is stable, $V^H$ is
a stable $G$-module and $H$ is the principal isotropy group of $V^H$.
\end{mainthm}

\begin{remark}\label{rem:rational-sings}
From the above,
Lemma \ref{lem:smooth=principal}  and
\cite[Theorem 2.11]{HerbigSchwarzSeaton3} we see that $N$
is a stable $G$-variety with
rational singularities.
\end{remark}

For the rest of this section we assume that $G$ is quasi-toral and that $V$ is a   faithful $1$-modular $G$-module with shell $N$.

Let $N_\sing$ denote the singular locus of $N$.
We say that \emph{the shell $N$ is minimal\/} or that \emph{the $G$-module $V$ is minimal\/} if $\codim_N N_\sing\geq 4$. Minimality  plays  an important role in this work and has many nice properties.
The first is the following, which is a consequence of Theorem \ref{thm:two}.

\begin{mainthm}\label{mainthm:4}
Suppose that
$N$ is   minimal.
Then $\dim G$ (as well as $\dim V$) is a topological invariant of the set of smooth points $(N\git G)_\sm$ of $N\git G$.
\end{mainthm}

Let $H$ be the isotropy group of a closed orbit in $N$. There is an associated  symplectic slice representation $(S,H)$ of $N$
(see \S \ref{sec:background}).
Then $S$ is a symplectic $H$-module which we may write as  $S^H\oplus S_0$ where $S^H$ and $S_0$ are symplectic $H$-modules. Moreover,  $S_0\simeq W_0\oplus W_0^*$ for a $1$-modular faithful $H$-module $W_0$. Let  $N_0$ denote the corresponding shell. We say that $(S,H)$ is of \emph{type O\/} if $\dim H>0$ and
$\dim N_0\git H=2$ (in which case $N_0\git H$ is well-known to be an orbifold, see   Corollary \ref{cor:reduce.N_0}).  We have the following two results (Proposition \ref{prop:codim}, Lemma \ref{lem:strata}, Corollaries \ref{cor:codim} and \ref{cor:matrixA}).

\begin{mainthm}\label{mainthm:5}
We always  have  $\codim_NN_\sing\geq 3$. There is equality if and only if  $N$ has a symplectic slice representation of type O.
\end{mainthm}

\begin{mainprop}\label{mainprop:1}
Let $G$ be a torus of rank
$\ell>0$ and let $V$ be a faithful $1$-modular $G$-module of dimension $n$ with associated matrix $A$ and shell $N$.  Then $N$ has a symplectic slice representation of type O if and only if the following holds.

\smallskip
\noindent $ (*) $  \quad There is an $r\in\N$, $1\leq  r\leq \ell$, and $n-r-1$ columns of $A$ of rank $\ell-r$.
\end{mainprop}

If $\codim_NN_\sing= 3$, then Theorem \ref{mainthm:4} does not hold; counterexamples are given in \cite{HerbigLawlerSeaton20}. However,
we have the following (Theorem  \ref{thm:main3}).

\begin{mainthm}\label{mainthm:6}
Let $G$ be quasi-toral and $V$ a $1$-modular faithful $G$-module with shell $N$. There is a linear subspace $V'\subset V$ with the following properties.
\begin{enumerate}
\item Let  $G'\subset G$ be the stabilizer of $V'$. Then $V'$ is a $1$-modular
 faithful  $G'$-module.
\item There is a $G'$-equivariant  inclusion $N':=N_{V'}\to N$ inducing
a Poisson
isomorphism $N'\git G'\simeq N\git G$.
\item $N'$ is minimal.
\end{enumerate}
If $V$ is a stable $G$-module, then $V'$ is a stable $G'$-module.
\end{mainthm}

We know that $\dim N=2\dim V-\dim G$ \cite[Proposition 3.2 (i)]{HerbigSchwarzSeaton2}.
Using  Theorem \ref{mainthm:4}, we see that $N$ is minimal if and only if  $\dim V$, $\dim G$ and $\dim N$ are minimal among triples $(V',G',N')$ with  $N'\git G'\simeq N\git G$. In particular, we have the following answer to the orbifold question.

\begin{maincor}\label{maincor:1}
Suppose that $N$ is minimal. The quotient $N\git G$ is an orbifold if and only if $G$ is finite.
\end{maincor}

Let $\Phi$ be an automorphism of $N$ and let $\sigma$ be an automorphism of $G$. We say that $\Phi\in\Aut(N)$ is \emph{$\sigma$-equivariant\/} if $\Phi\circ g=\sigma(g)\circ \Phi$ for all $g\in G$. Echoing \cite[Theorem 1.12, Remark 1.15]{GWSQuotients} we have the following two results (Corollary \ref{cor:lifting}, Proposition \ref{prop:lifting} and Theorem \ref{thm:isomorphism}).
\begin{mainthm}\label{mainthm:2}
Suppose that   $N$ is minimal.
Let $\phi$ be an algebraic automorphism of $N\git G$. Then $\phi$ lifts to an algebraic automorphism  of $N$ and any such lift is $\sigma$-equivariant for some $\sigma\in\Aut(G)$.
\end{mainthm}

\begin{mainthm}\label{mainthm:3}
For $i=1$, $2$, let $G_i$ be quasi-toral and $V_i$ a faithful $1$-modular
minimal $G_i$-module with shell $N_i$.
If
$$
N_1\git G_1\simeq N_2\git G_2
$$
as affine varieties, then there is a linear isomorphism
$$
\Gamma\colon V_1\oplus V_1^*\tosim V_2\oplus V_2^*
$$
inducing  isomorphisms of $N_1$ and $N_2$ and of $G_1$ and $G_2$.
\end{mainthm}

In \S \ref{sec:real} we establish analogues of our results for unitary representations $V$ of compact quasi-toral Lie groups $K$, their associated real shells and quotients.

\section*{Acknowledgements}
C.S.~was supported by a Rhodes College Faculty Development Grant from Rhodes College and would like to thank Yi Song, whose undergraduate research project was related to the early stages of this investigation.

\section{Background}\label{sec:background}

For the following see   \cite{LunaSlice}.
Let $G$ be a reductive complex Lie group and $X$ an affine $G$-variety.  The algebra of invariants $\C[X]^G$ is a finitely generated domain with corresponding variety $X\git G$. Dual to the inclusion  $\C[X]^G\to\C[X]$ we have the quotient mapping $\pi_X$ (or just $\pi$) $\colon X\to X\git G$. The variety $Z=X\git G$ parameterizes the closed $G$-orbits in $X$, i.e., each fiber of $\pi$ contains a unique closed orbit.   Let $x\in X$ such that $Gx$ is closed. Then the isotropy group $H=G_x$ is reductive. Let $Z_{(H)}$ denote the points of $Z$ whose corresponding closed orbits have isotropy group in the conjugacy class $(H)$ of $H$. Let $X^{(H)}$ denote the inverse image of $Z_{(H)}$ in $X$. The subvarieties $Z_{(H)}$ are called the \emph{strata of $Z$}. There is a (unique) dense open stratum $Z_\pr$,  the \emph{principal stratum},  with inverse image denoted $X_\pr$ \cite[\S 3]{LunaRichardson}. The closed orbits in $Z_\pr$ are called \emph{principal orbits\/}. If $Z$ is a stable $G$-variety, then all orbits in $Z_\pr$ are closed, i.e., principal. A subset of $X$ is \emph{$G$-saturated\/} if it is the inverse image of a subset of $Z$.

See \cite[Section 2]{HerbigSchwarzSeaton3} for the following. Let $V$ be a faithful $G$-module.
 Let $V_{(r)}=\{v\in V\mid \dim G_v=r\}$. For $k\in \N$ we say that $V$ is \emph{$k$-modular\/}   if $\codim_V V_{(r)}\geq r+k$ for $1\leq r\leq\dim G$. The shell $N$ is a
 complete intersection
 variety if and only if $V$ is $1$-modular, which we now assume.  The null cone $\NN(V)$ (resp.\ $\NN(N)$) is $\pi_V\inv(\pi_V(0))$ (resp.\    $\pi_N\inv(\pi_N(0)))$.  Luna's slice theorem \cite{LunaSlice} shows that the strata of $V\git G$ are smooth and locally closed, and the symplectic slice theorem \cite[Theorem 3.15]{HerbigSchwarzSeaton2}) shows the same thing for the strata of $N\git G$. Let $Gx$ be a closed orbit in $N$ with isotropy group $H=G_x$. There is an associated \emph{symplectic slice representation\/} $(S,H)$, uniquely determined by $H$. There is an $H$-module $W$ (not usually unique) such that  $S\simeq W\oplus W^*$. Let  $N_S$ denote  the  corresponding shell. We have a decomposition  $W=W^H\oplus W_0$
 of $H$-modules so that $N_S\simeq S^H\times N_0$ where $N_0$ is the shell of $W_0$. We say that $(S,H)$   is \emph{proper\/} if $H\neq G$. We say that   $(S,H)$ is \emph{maximal\/} if it is proper and for any other symplectic slice representation $(S',H')$ where $H\subset H'$, either $H'=H$ or $H'=G$.

For $k\geq 0$, we say that $V$ is \emph{$k$-principal} if $\codim  V\setminus V_{\pr} \geq k$.
If $V$ has finite principal isotropy groups, we abbreviate by saying that $V$ \emph{has FPIG}.  If the principal isotropy groups are trivial, we say that $V$ \emph{has TPIG}. We say that $V$ is \emph{$k$-large\/} if it is $k$-principal, $k$-modular and has FPIG.

The singular points of $N$ (resp.\ $N\git G$) are denoted by $N_\sing$ (resp.\ $(N\git G)_\sing$) and the smooth points of $N\git G$ are denoted by $(N\git G)_\sm$.
The standard symplectic structure on $V\oplus V^*$ induces a Poisson algebra structure on $\C[N]$ and on $\C[N]^G\simeq \C[N\git G]$. See \cite[\S 2.1]{HerbigSchwarzSeaton3} and
\cite[discussion after Remark 2.4]{HerbigSchwarzSeaton2}.

\begin{lemma}\label{lem:strata}
Assume that $V$ is a $1$-modular $G$-module with shell $N$. Let  $(S,H)$ be a symplectic slice representation of $N$.
Then
\begin{enumerate}
\item $\codim_NN^{(H)}=\codim_{N_0}\NN(N_0).$
\item $\codim_N (N_\sing\cap N^{(H)})=\codim_{N_0}((N_0)_\sing \cap\NN(N_0)).$
\end{enumerate}
\end{lemma}

\begin{proof}
 It follows from the symplectic slice theorem \cite[Corollary 3.16]{HerbigSchwarzSeaton2} that
 $$
 \codim_NN^{(H)}=\codim_{N_S}(N_S^{(H)}) \text { and }\codim_N(N^{(H)})_\sing=\codim_{N_S}(N_S^{(H)})_\sing.
 $$
 Now $N_S=S^H\times N_0$ and $N_S^{(H)}=S^H\times\NN(N_0)$ giving (1). Similarly,
 $(N_S^{(H)})_\sing =S^H\times((N_0)_\sing\cap\NN(N_0))$  giving (2).
\end{proof}

Note that if $G$ is quasi-toral and $V$ is
 faithful and
 stable (equivalently,
  $1$-large \cite[Theorem 3.2]{HerbigSchwarz}),
then the principal isotropy group of the $G^0$-module $V$ is the kernel of the $G^0$-action.
In  \cite{HerbigSchwarzSeaton2}, Corollary
3.16, Proposition 5.4 and the discussion after  Lemma 5.6 show the following.

\begin{lemma}\label{lem:smooth=principal}
Suppose that $G$ is quasi-toral and $V$ is a faithful stable $G$-module with shell $N$.
\begin{enumerate}
\item The $G$-module $V$ is $1$-large.
If $G$ is a torus, then $V$ has TPIG.
\item The shell $N$ has rational singularities, in particular, it is normal.
\item Let $(S,H)$ be a symplectic slice representation of $N$ with Lagrangian $H$-submodule $W$. Then
$W$ is $1$-modular and $N_S$ has rational singularities.
\item $(N\git G)_\sm=(N\git G)_\pr$.
\end{enumerate}
\end{lemma}

Assume $G$ is quasi-toral. Choose an isomorphism $G^0 \simeq (\C^\times)^\ell$ and a basis
$\{e_1,\dots,e_n\}$ of $V\simeq\C^n$ with respect to which $G^0$ acts diagonally.
The action of $G^0$ on $V$ is determined by a weight matrix $A\in\Z^{\ell\times n}$.
Specifically, the action of $(t_1,\ldots,t_\ell)\in G^0$ on
$e_j$ is given by the character
$\prod_{i=1}^\ell t_i^{a_{i,j}}$.
The dimension of the isotropy group of a point
$(v_1,\dots,v_n)$ is given by the corank of the submatrix of $A$ corresponding
to the columns $A_i$ such that
$v_i\neq 0$. Hence,
$(v_1,\dots,v_n)\in V_{(r)}$ if and only if the corresponding submatrix has rank $\ell - r$
and  $V$ is $k$-modular if and only if every $\ell\times(n-k-r)$ submatrix of $A$ has rank at least $\ell-r$. This is true if
and only if
every $\ell\times(n-k)$ submatrix of $A$ has rank $\ell$.
We therefore have the following.

\begin{lemma}\label{lem:k-modular-torus}
Let $G$ be quasi-toral and $V$ a faithful $G$-module. Then $V$ is $k$-modular if and only if every $\ell\times(n-k)$ submatrix of $A$ has rank $\ell$.
\end{lemma}
\begin{remark}\label{rem:general}
If  $G$ is a torus, then $A$ can be arbitrary. Hence the  ``general''  $G$-module of dimension $n$ is $k$-modular if $n-k\geq\ell$. For semisimple $G$, the ``general'' $G$-module
is $k$-modular   \cite[Corollary 11.6]{GWSlifting},
\cite[Theorem 3.6]{HerbigSchwarzSeaton2}.
\end{remark}

\begin{remark}\label{rem:1-modular}
In the case that $k=1$, the criterion of Lemma \ref{lem:k-modular-torus} is equivalent to the existence of a relation $\sum_i a_iA_i=0$ where the $a_i\in\Z\setminus\{0\}$ and $A_i$ is the $i$th column
of the matrix $A$.
To see this, note that $k$-modularity depends only on the action of $G^0$ so that we may assume $G$ is a torus.
UTCLS, which also does not change $k$-modularity, we may assume that the $a_i$ are positive, and then there is a $G$-invariant monomial
$\prod_{i=1}^n x_i^{a_i}$.
By \cite[Lemma 2]{WehlauPopov}
and Lemma \ref{lem:smooth=principal}(1),
 $V$ is stable and  $1$-large.
\end{remark}

\section{Stable actions and symplectic slices}
Let $V$ be a $G$-module where $G$ is quasi-toral.
It follows from
\cite[Lemma 2]{WehlauPopov} that
$V$ is
stable if and only if there is
no proper $G$-submodule $V'\subset V$ such that restriction to $V'$ gives an isomorphism
$\C[V]^G\tosim\C[V']^G$.
Let $n=\dim V$ and $\ell = \dim G$.
For the rest of this section  we assume that $G$ is commutative.

\begin{lemma}\label{lem:closed-orbit-V}
Let  $V$ be a $G$-module.  If the shell $N$ of $V$ contains a nonzero closed $G$-orbit, then UTCLS, $V$ contains a nonzero stable $G$-submodule.
\end{lemma}

\begin{proof}
We may assume that $G$ acts faithfully on $V$ and that $V^G=(0)$.
Choose a basis $\{v_1,\dots,v_n\}$ for $V$ and corresponding coordinate functions $\{x_1,\dots,x_n\}$ such that $G$ acts on $x_j$ by the character $\chi_j$.
Clearly, $n\geq \ell$, and if there is equality, then
there are no nonzero closed orbits in $N$. Thus $\ell<n$ and we may assume that there is a nontrivial relation $\sum_{i=1}^{\ell+1} a_i\chi_i=0$ where the $a_i\in\Z$. By replacing some of the $\chi_i$ by their inverses,
corresponding to changing the Lagrangian section,
we may assume that $a_i\geq 0$ for all $i$.
By renumbering we may assume that $a_i>0$ for $i=1,\dots,p$ and that $a_i=0$ for $i>p$. Then $\prod_{i=1}^p x_i^{a_i}$ is $G$-invariant and the span of $v_1,\dots,v_p$    is  stable.
\end{proof}

For the rest of this section we assume that  $V$ is faithful and $1$-modular.

\begin{proposition}\label{prop:make.stable}
Let $G$ and $V$ be as above.
\begin{enumerate}
\item  UTCLS, $V$ is stable.
\item Suppose that $V'$ is a  stable $G$-submodule of $V$. Then, UTCLS, $V$ is stable and $V'$ remains a submodule of $V$.
\end{enumerate}
\end{proposition}

\begin{proof}
We may assume that $V^G=(0)$.
Since $V$ is $1$-modular, $n>\ell$, and $N$ contains nonzero closed orbits. By Lemma \ref{lem:closed-orbit-V}, UTCLS, there is a nonzero stable $G$-submodule $V'\subset V$. Now (1) follows if we can show (2).

Let $H$ denote the principal isotropy group of $V'$. Since $G$ is commutative, $H$ acts trivially on $V'$. Let $V''$ be a $G$-complementary subspace to $V'$.   Then $(V'',H)$ is $1$-modular since it is a slice representation  (up to trivial factors) of $(V,G)$ \cite[Remarks 9.6]{GWSlifting}. By induction on dimension we may assume that, UTCLS, $(V'',H)$ is stable, hence $V=V'\oplus V''$ is stable.
\end{proof}

\begin{remark}\label{rem:H-fixed-points}
Suppose that $V'\subset V$ are stable $G$-modules where $H$ is the principal isotropy group of $V'$. Then $V^H\supset V'$ is  stable with principal isotropy group $H$.
\end{remark}

\begin{remarks}
\begin{enumerate}
\item Let $V$ be $0$-modular and $G$ any reductive group.
Then
$N=N_V$ is a complete intersection \cite[Proposition 3.2]{HerbigSchwarzSeaton2} and
$V_{(0)}\subset V$ is the set of points with finite isotropy group.
If
$N$
is irreducible, then \cite[Theorem 2.2]{HerbigSchwarz} shows that
$N$
is the closure of $(V_{(0)}\times V^*)\cap
N$,
and  the latter set consists of points with finite isotropy.
Such points are smooth points of $N$ \cite[Proposition 3.2]{HerbigSchwarzSeaton2}. Thus the smooth points of $N$ are dense, so it is reduced.
\item
Assume that  $G\to\GL(V)$ is a faithful torus representation. Then $V$ is $0$-modular.
Bulois \cite[Theorem 2.2]{Bulois} shows that $N$ is normal iff it is irreducible.     If $N$ is irreducible, then it is reduced by   (1)  so that $V$ is $1$-modular  and has rational singularities by Remark \ref{rem:rational-sings}. Hence $N$ is normal. The converse is, of course, trivial.
\end{enumerate}
\end{remarks}

\begin{lemma}\label{lem:i.groups}
Let $(S,H)$ be a maximal proper symplectic slice representation of $N$. Then, UTCLS, $V$ is stable and $H$ is the principal isotropy group of a stable $G$-submodule of $V$.
\end{lemma}

\begin{proof}
As an $H$-module,
$$
U=V\oplus V^*\simeq S_0\oplus U^H
$$
where  $G$ preserves $S_0$ and $U^H$.  We can decompose $S_0$ as $W_0\oplus W_0^*$ and $U^H$ as $U_0\oplus U_0^*$ where each factor is $G$-stable, hence a direct sum of weight spaces.
Since $H$ is an isotropy group of a nonzero closed orbit in $N$, the orbit lies in  $N_{U_0}\times\{0\}$ where $\{0\}$ is the origin of $S_0$. By Lemma  \ref{lem:closed-orbit-V},  UTCLS,   $U_0$ contains a nonzero stable $G$-submodule $U_0'$.    Let $H'$ be the principal  isotropy group of   $U_0'$. Then $H'\supset H$ is the isotropy group of a closed orbit in $N$ and by maximality, $H'=H$.
By Proposition \ref{prop:make.stable}(2) we may assume that $U_0'\subset V$ where $V$ is stable.
\end{proof}

 The following
implies
 Theorem \ref{mainthm:1}.

\begin{theorem}\label{thm:i.groups}
Let $G$ be commutative   reductive and $V$ a $1$-modular faithful  $G$-module with shell $N$. Let $H$ be an isotropy group of  a nonzero closed orbit in $N$. Then, UTCLS, $V$ is stable, $V^H$ is stable (as a $G$-module) and $H$ is the principal isotropy group of $V^H$.
\end{theorem}

\begin{proof}
Recall that any isotropy group of a closed orbit in  $N$ determines the corresponding symplectic slice representation. There is a maximal sequence of isotropy groups  of closed orbits in $N$:
$$
H=H_r\subsetneq H_{r-1}\subsetneq \cdots H_0\subsetneq G.
$$
By Lemma \ref{lem:i.groups}, we may assume that $H_0$ is the principal isotropy group of a stable $G$-submodule $V_0\subset V$. Let $V'$ be a $G$-complement to $V_0$ in $V$. Then the slice representation of $H_0$, up to trivial factors, is its action on $V'$.   The $H_i$ for $i>0$ are isotropy groups of closed orbits in $N_{V'}$ for the action of $H_0$. By induction on $r$, UTCLS,  $H_r=H$ is the principal isotropy group of a stable $H_0$-submodule $V''$ of $V'$, hence the principal isotropy group of the stable $G$-submodule $V_0\oplus V''\subset V$. By Remark \ref{rem:H-fixed-points}, $V^H$ is stable with principal isotropy group $H$.  By Proposition \ref{prop:make.stable} we can arrange that $V$ is stable.
\end{proof}

\section{Codimension of singularities and slice representations of type O}
 Let $G$ be a reductive complex group and $V$ a $1$-modular $G$-module with shell $N$
 and symplectic quotient $M=N\git G$.
 We give some results about this general situation before returning to the case that $G$ is quasi-toral.  Let $(S=W\oplus W^*,H)$ be a symplectic slice representation of $N$.
As in \S \ref{sec:background}, we have $W=W^H\oplus W_0$ as $H$-module and $N_S\simeq S^H\times N_0$ where $N_0$ is the shell of $W_0$.
Let $H_1$, $H_2$ be two subgroups of $G$.  We write $(H_1)=(H_2)$ (resp.\ $(H_1)\leq (H_2)$) if $H_1$ is conjugate to  (resp.\   a subgroup of) $H_2$.

 From the symplectic slice theorem one gets the following.

\begin{lemma}\label{lem:codim}
Let $G$, etc.\ be as above.
\begin{enumerate}
\item $(N\git G)_{(H)}$ has codimension $\dim N_0\git H$ in $N\git G$.
\item Let $x\in N$ such that $Gx$ is closed. Then there is an affine $G$-saturated neighborhood $U$ of $x$ such that $(G_y)\leq (G_x)$ for all $y\in U$.
\end{enumerate}
\end{lemma}

Let $N^{<H>}$ denote $N^{(H)}\cap N^H$. The proof of \cite[Lemma 5.5]{GWSliftingHomotopies} gives the following.

\begin{lemma}\label{lem:general.stuff}
Let $(S,H)$, etc.\ be as above.
\begin{enumerate}
\item $N^{<H>}$ is Zariski open in $N^H$ and all $G$-orbits intersecting $N^{<H>}$ are closed.
\item  $\overline{M_{(H)}}=\bigcup_{(H')\geq(H)} M_{(H')}=\pi(N^H)$.
\item If $g\in G$ and $gN^{<H>}\cap N^H\neq\emptyset$, then $g\in N_G(H)$.
\item $M_{(H)}$ is the quotient of $N^{<H>}$ by the free action of $N_G(H)/H$. Hence $N^{<H>}\to M_{(H)}$ is a principal $N_G(H)/H$-bundle \cite[Corollaire 1]{LunaSlice}.
\end{enumerate}
\end{lemma}

\begin{theorem}\label{thm:strata.connected}
Let $G$, $H$, etc.\ be as above. Suppose that $V^H$ is a $1$-modular $N_G(H)/H$-module. Then $M_{(H)}$ is irreducible.
\end{theorem}

\begin{proof}
Let $N_{V^H}$ denote the shell of $V^H$ for the action of $N_G(H)/H$. By Lemma \ref{lem:general.stuff}, $N^{<H>}$ is just the open set of principal orbits in $N_{V^H}$. Since $N_{V^H}$ is irreducible, so is $N^{<H>}$. Thus $M_{(H)}$ is irreducible.
\end{proof}
Applying Theorem \ref{thm:i.groups} we obtain the following.
\begin{corollary}\label{cor:strata-connected}
If $G$ is a torus, then every stratum of $N$ is irreducible.
\end{corollary}

The corollary fails if $G$ is not a torus.

\begin{example}\label{ex:non-connected stratum}
Let $W$ be the  irreducible two-dimensional representation of $\SL_2$ and  let $\nu_i$ be the irreducible  $\C^\times$-module of weight $i$, $i\in\Z$. Let $G=\SL_2\times\C^\times$ and let
$V=2W\oplus  (\nu_1+\nu_{-1})\otimes  W$.
It is easy to see that $V$ is a $1$-modular (and stable) $G$-module so that the shell $N$ is a complete intersection variety. Let $H=\C^\times$. Then $H$ is the principal isotropy group of $V^H\simeq 2W$ and the corresponding symplectic slice representation is $2\nu_0+4(\nu_1+\nu_{-1})$. Hence $X=(N\git G)_{(H)}$ has dimension $2$.

We claim that $X$ has two irreducible components. Now $N^H$ (with its $\SL_2$-action) is the shell of $2W$ and the image of $N^H$ in $N\git G$ is the closure of $X$. By \cite[\S 3]{Becker}  $N^H\git\SL_2$ has two irreducible components which are normal of dimension $2$. Thus $X$ has two irreducible components.
\end{example}

 \begin{remark}\label{rem:bigger,example}
 Let $V=2W\oplus n(\nu_1+\nu_{-1})\otimes W$ with $n\in\N$. We can make $V$ $k$-modular for $k$ arbitrarily large by choosing $n$ large, yet $(N\git G)_{(H)}$ always has two irreducible components.
 \end{remark}

 Now   we assume that $G$ is quasi-toral.

\begin{definition}
We say that a symplectic slice representation $(S=W\oplus W^*,H)$   of $N$
is \emph{of type O} if $\dim H>0$ and $\dim N_0\git H=2$.
\end{definition}

\begin{proposition}\label{prop:codim}
Let $G$ be quasi-toral of rank $\ell>0$. Let $V$ be a
$1$-modular faithful  $G$-module of dimension $n$ with shell $N$ where $V^G=(0)$.
\begin{enumerate}
\item The null cone $\NN(V\oplus V^*)$ is a subset of $N$, hence equals $\NN(N)$.
\item If $n-\ell\geq 2$, then $\codim_N(N_\sing\cap \NN(N)) \geq 4$.
\item If $n-\ell=1$, then $\codim_N(N_\sing\cap \NN(N))\geq  3$. In fact, there is equality.
\end{enumerate}
\end{proposition}

\begin{proof}  By \cite[Proposition 3.2]{HerbigSchwarzSeaton2}, $N_\sing=\{x\in N\mid \dim G_x>0\}$. We may assume that
$G=G^0$ and that $V$ is stable.
Part (1) is clear since the entries of $\mu$ are homogeneous quadratic
invariants.
Let $v_1,\dots,v_n$ be a basis of $V$ corresponding to characters $\chi_1,\dots,\chi_n$ of $G$. The characters of $V^*$ are $\chi_1\inv,\dots,\chi_n\inv$. We know that
the null cone
$\NN(V\oplus V^*)$ is a finite union of linear subspaces $Z_\lambda:=V_\lambda\oplus V^*_\lambda$ where $\lambda\colon \C^\times\to G$ is a homomorphism. Here $V_\lambda$ denotes the direct sum of the strictly positive weight spaces of $\lambda$ acting on $V$ and similarly for $V^*_\lambda:=(V^*)_\lambda$.  Note that if $\lambda$ is not the constant homomorphism (which we may assume), then $V_\lambda\neq (0)$.   WLOG assume that $V_\lambda$ is the span of $v_1,\dots,v_k$.  Then $V^*_\lambda$ is at worst the span of the $n-k$ basis elements $\xi_{k+1},\dots,\xi_n$ corresponding to $\chi_{k+1}\inv,\dots,\chi_n\inv$.

 Let $(v,\xi)\in Z_\lambda\cap N_\sing$.  Since $V$ is stable, for any $1\leq i\leq n$ there is a monomial invariant $\prod x_j^{a_j}$ where all $a_j$ are nonnegative and $a_i>0$.   Thus  any cardinality $n-1$ subset of $\{\chi_1,\dots,\chi_n\}$ has finite joint kernel. and
 $$
 (v,\xi)\in Z_\lambda\cap N_\sing
 $$
   only if at least two  of $v_1,\dots,v_k,\xi_{k+1},\dots,\xi_n$ are zero. Hence
 $$
 \codim_N (N_\sing\cap \NN(N))\geq \dim N-(n-2)=2n-\ell-n+2=(n-\ell)+2.
 $$
This immediately gives the inequalities of (2) and (3). If $n-\ell=1$, then $\dim \C[V]^G=1$ and the generating invariant monomial involves all the coordinate functions $x_1,\dots,x_n$. Thus, additively, there is a relation
$$
\sum_{i=1}^n a_i\chi_i=0
$$
and there are no other linearly independent relations since $\dim G=n-1$. Thus no subset of cardinality $n-2$ of $\{\chi_1,\dots,\chi_n\}$ has finite joint kernel and there are $\lambda$ such that $  Z_\lambda\cap  N_\sing$ has dimension $n-2$.
It follows that $\codim_N(N_\sing\cap \NN(N))=3$.
\end{proof}

Using Lemmas \ref{lem:strata} and \ref{lem:codim} we obtain the following.
 \begin{corollary}\label{cor:codim}
Let $G$ be quasi-toral and $V$ a faithful $1$-modular $G$-module with shell $N$. The following are equivalent.
 \begin{enumerate}
\item $\codim_NN_\sing\geq 4$.
\item $N$ has no symplectic slice representation of type O.
\item If $(N\git G)_{(H)}$ is a stratum of codimension $2$, then $H$ is finite.
\end{enumerate}
 \end{corollary}

 \begin{corollary}\label{cor:type.O.G^0.vs.G}
 The $G$-variety $N$ has slice representations of type O if and only if the same is true for $N$ as $G^0$-variety.
 \end{corollary}

\begin{lemma}\label{lem:H-type-O} Let $G$ be a torus and $V$ a $1$-modular faithful $G$-module. Let $H$ be the isotropy group of a nonzero closed orbit in $N$.   Let $W$ be a $G$-complement to $V^H$ in $V$. Then $H$ is the isotropy group of  a slice representation of type O if and only if $\dim W=\dim H+1$.
 \end{lemma}

 \begin{proof}
 By Theorem  \ref{thm:i.groups}, we may assume that $V$ is stable and that $V^H$ is   a stable $G$-submodule with principal isotropy group $H$. The slice representation of $H$, up to trivial factors, is its action on $W$, which is stable since $V$ is stable. It follows that the symplectic slice representation of $H$ has type O if and only if $\dim W=\dim H+1$.
 \end{proof}

 Let $H$ be a subtorus of $G$ such that $V=V^H\oplus W$ as a $G$-module where $W$ is a $1$-modular  $H$-module such that $\dim W=\dim H+1$.    We say that \emph{$(W,H)$ is maximal\/} if, UTCLS,
 there is no subtorus $H'$ of $G$ and decomposition $V=V^{H'}\oplus W'$ where $W'$ is a $1$-modular $H'$-module,  $\dim W'=\dim H'+1$,  $H\subset H'$, $W\subset W'$ and $\dim H'>\dim H$.

\begin{lemma}
In the situation of Lemma \ref{lem:H-type-O}, suppose that $\dim W=\dim H+1$. Then $(W,H)$ is maximal.
\end{lemma}

\begin{proof}
We may assume that $W$ is a stable $H$-module. If $(W,H)$ were not maximal, then by Lemma \ref{lem:general.stuff}, some $M_{(H')}$ would lie in the boundary of  $M_{(H)}$ where both have codimension two in $M$. This  is not possible.
\end{proof}

\begin{corollary}\label{cor:maximal}
Let $H$ be a subtorus of $G$ such that $V=V^H\oplus W$ as a $G$-module where $W$ is a $1$-modular  $H$-module such that $\dim W=\dim H+1$. Then the following are equivalent.
\begin{enumerate}
\item Up to trivial factors,  $(W\oplus W^*,H')$ is a symplectic slice representation of $N$ of type O where $H'$ is a finite extension of $H$.
\item $(W,H)$ is maximal.
\end{enumerate}
\end{corollary}

\begin{proof} We only have to show that (2) implies (1).
Let $\rho\colon G\to  \GL(W)$ give the $G$-module structure on $W$ and let  $K=\Ker\rho$.  We may assume that $V$ is stable. Then $V^H$ is a stable $K$-module. If $\dim G/H=\dim K$, then $V^H$ is a stable $G$-module whose principal isotropy group has dimension $\dim H$, and (1) follows. Otherwise, $\dim G/H=\dim K+1$ and we may choose a $1$-dimensional subtorus $R\subset G$ such that the product mapping $K\times R\to G/H$ has finite kernel and such that $R$ acts on $W$ by a character $\chi\neq 1$.

Since $V^H$ is a stable $K$-module with finite principal isotropy group, the action is   $1$-modular.   Let $V^H_{(r)}$ denote the points of $V^H$ whose $K$-isotropy group has dimension $r$. Since $G$ is commutative, the closure $\overline{V^H_{(r)}}$ is a finite union of $G$-submodules of $V^H$. If $(V^H,K  R)$ is not $1$-modular, then  for some $r\geq 0$,  $R$ fixes a linear subspace $L\subset \overline{V^H_{(r)}}$ of codimension $r+1$ in $V^H$.   Let $P$ be a $G$-complement to $L$ in $V^H$. Let $K_1=\Ker(K\to\GL(L))$ . Then $\dim K_1=r$, $\dim P=r+1$ and  $K_1$
acts stably on $P$ with a quotient of dimension $1$. Thus, by perhaps changing $W$ to $W^*$, $K_1 R  H$ acts trivially on $L$ and stably on its complement $P\oplus W$ with a quotient of dimension $1$. This   contradicts maximality of $(W,H)$. Hence $(V^H,G/H)$ is $1$-modular and we may assume that it is stable. We can assume that  the principal isotropy group $H'$ of $V^H$ acts stably on $W$ in which case we have $\dim H=\dim H'$ and (1) holds.
\end{proof}

Let $G$ be a torus and  $V$ a $1$-modular faithful $G$-module. Suppose that  $H$ is the isotropy group of a slice representation of type O. Then $V^H$ is a $1$-modular $G/H$-module.
The following is Proposition \ref{mainprop:1}.

\begin{corollary}\label{cor:matrixA}
Let $G$ be a torus of rank $\ell>0$ and let $V$ be a faithful $1$-modular $G$-module of dimension $n$ with associated matrix $A$ and shell $N$.  Then $N$ has a symplectic slice representation of type O if and only if the following holds.

\smallskip
\noindent $ (*) $  \quad There is an $r\in\N$, $1\leq  r\leq \ell$, and $n-r-1$ columns of $A$ of rank $\ell-r$.
\end{corollary}

\begin{proof}
Since $V$ is $1$-modular, $n>\ell$. We may assume that $V$ is stable which only requires changing the sign of some of the columns of the $\ell\times n$ matrix $A$. Let $\chi_1,\dots,\chi_n$ be the characters of $G$ corresponding to the columns of $A$.  If $(*)$ holds, we may assume that the joint kernel $H$ of $\chi_1,\dots,\chi_{n-r-1}$ has rank  $r\geq 1$. Then $H$ acts stably  on the  subspace $W$ of $V$  spanned by the last $r+1$ basis elements. We may assume that $(W,H^0)$ is maximal, in which case  Corollary \ref{cor:maximal} shows that $N$ has a slice representation of type O. Conversely, if $N$ has a symplectic slice representation of type O, then  Lemma \ref{lem:H-type-O} shows that  $A$ satisfies   $(*)$.
\end{proof}

\section{Determining $\dim G$ topologically}

 Let $V$ be a $1$-modular  faithful $G$-module where $G$ is quasi-toral. Let $N$ be the shell and let $M$ denote $N\git G$.
\begin{lemma}\label{lem:codim.Npr}
Suppose that $N$ is minimal.
\begin{enumerate}
\item $\codim_N(N\setminus N_\pr)\geq 2$, i.e., $N$ is $2$-principal.
\item For $i=1$, $2$, we have an equality of homotopy groups $\pi_i(N_\pr)=\pi_i(N_\sm)$.
\end{enumerate}
\end{lemma}

\begin{proof}
Let $(S,H)$ be a symplectic slice representation of $N$ which is not principal. By Lemma \ref{lem:strata}   we have
$$
\codim_{N}(N^{(H)})=\codim_{N_0}\NN(N_0)
$$
where $\NN(N_0)\subset N_0$ since $G^0$ is a torus.
Now $N_0=W_0\oplus W_0^*$ where $W_0$ is a faithful stable $H$-module. If $H$ is finite, then $\NN(N_0)=\{0\}$ and its codimension in $N_0$ is $2\dim W_0\geq 2$.
If $\ell=\dim H>0$, let $\dim W_0=n$. Then by Proposition \ref{prop:codim} and Corollary \ref{cor:codim}, $n-\ell\geq 2$ and
$$
\codim_{N_0}\NN(N_0)\geq 2n-\ell-n=n-\ell\geq 2.
$$
This gives (1). Now (2) follows from \cite[Lemma 2.4]{HerbigSchwarzSeaton} since $N_\sm$ is smooth and
$$
\dim (N_\sm\setminus N_\pr)\leq\dim(N\setminus N_\pr)\leq \dim N-2=\dim N_\sm-2.
\qedhere
$$
\end{proof}

   For an abelian group $H$,   define   $\rank H$ to be $\dim_\Q (H\otimes_\Z\Q)$.
 Let $N^*:=N\setminus\{0\}$ and let $X:=N^*/\C^\times\subset \P:=\P(V\oplus V^*)$. Let $Z:=(N_\sing\setminus\{0\})/\C^\times\subset X$. Since $\codim_NN_\sing\geq 4$ and $N^*$ and $N_\sing$ are $\C^\times$-stable, $\codim_X Z\geq 4$.  Let $\ell=\dim G$.

 \begin{lemma}
The inclusion $X\smallsetminus Z\to X$ induces   isomorphisms $\pi_1(X\smallsetminus Z)\tosim\pi_1(X)$
and   $\pi_2(X\smallsetminus Z)\tosim\pi_2(X)$.
\end{lemma}
\begin{proof}
We follow \cite{Starr}. Choose a generic linear section
$L$
of $\mathbb{P}$  which has codimension $\dim  X-3$
and does not intersect $Z$. Then \cite[Theorem 1.2]{GM} (with $\hat{n} =3$) implies that
the maps $\pi_i(L\cap X)\to\pi_i(X)$ and
\[
    \pi_i(L\cap (X\smallsetminus Z)) = \pi_i(L\cap X)\to \pi_i(X\smallsetminus Z).
\]
are isomorphisms for $i=1$, $2$.
\end{proof}

 \begin{theorem}\label{thm:two}
 If $N$ is minimal, then $\ell=\rank \pi_2(M_\sm)-\rank H$ where $H$ is the kernel of a
 surjective homomorphism   $\pi_1(M_\sm)\to\pi_0(G)$.
  \end{theorem}

\begin{proof}
We can replace $M_\sm$ by $M_\pr$ since they are the same
by Lemma \ref{lem:smooth=principal}(4) and Proposition \ref{prop:make.stable}(1).
Using  \cite[Corollary 9.7]{FL} as in the proof of \cite[Theorem 3.22]{HerbigSchwarzSeaton2} we get that $\pi_2(X)=\Z$ and $\pi_1(X)=0$. From the fibration $\C^\times\to N^*\to X$ we get an exact sequence
\[
  \underbrace{\pi_2(\C^\times)}_{=0} \to \pi_2(N^*)\to \underbrace{\pi_2(X)}_{=\Z}\to\underbrace{\pi_1(\C^\times)}_{=\Z}
    \to\pi_1(N^*)\to \underbrace{\pi_1(X)}_{=0},
\]
It follows that $\pi_1(N^*)$ is abelian and that  $\pi_1(N^*)$ and $\pi_2(N^*)$ are either both finite or both of rank $1$.
From the diagram of fibrations
\[
\begin{CD}
\C^\times@>>>N_\sm@>>>X\smallsetminus Z \\
@VVV @VVV @VVV \\
\C^\times @>>>N^*@>>>X
\end{CD}
\]
we see that $\pi_1(N_\sm)\simeq\pi_1(N^*)$ and $\pi_2(N_\sm)\simeq\pi_2(N^*)$. Thus $\pi_i(N_\pr)\simeq\pi_i(N^*)$ for $i=1$, $2$ and $\pi_1(N_\pr)$ and $\pi_2(N_\pr)$ are either both finite or both of rank $1$.

From the fibration $G\to N_\pr\to M_\pr$ we get an exact sequence
\[
  \underbrace{\pi_2(G)}_{=0} \to \pi_2(N_\pr)\to \pi_2(M_\pr) \to\underbrace{\pi_1(G)}_{=\Z^\ell}
    \to\pi_1(N_\pr)\to \pi_1(M_\pr)\to\pi_0(G)
\]
Now we do all calculations
modulo finite groups to only consider the rank of the abelian groups.
Let $H$ denote the image of $\pi_1(N_\pr)\to \pi_1(M_\pr)$ which is the same as the inverse image of the identity of $\pi_0(G)$ under $\pi_1(M_\pr)\to\pi_0(G)$.
We have
$$
\rank \pi_2(N_\pr)-\rank\pi_2(M_\pr)+\ell-\rank \pi_1(N_\pr)+\rank H=0.
$$
Since $\pi_1(N_\pr)$ and $\pi_2(N_\pr)$ have the same rank,
$$
-\rank\pi_2(M_\pr)+\ell+\rank H=0,
$$
establishing the theorem.
\end{proof}

\begin{corollary}\label{cor:rank.G}
Let $G$ be quasi-toral and $V$ a
$1$-modular
 faithful
minimal $G$-module with shell $N$.
Let $G'$, $V'$ and $N'$ have the same properties.  If
$N\git G\simeq N'\git G'$,
then $\dim G=\dim G'$,
 $\dim V=\dim V'$ and $\dim N=\dim N'$.
\end{corollary}

Corollaries \ref{cor:codim} and \ref{cor:type.O.G^0.vs.G} give the following.

\begin{corollary}\label{cor:rank G^0}
 Let $G$ be quasi-toral and $V$ a
 $1$-modular
  faithful minimal
 $G$-module with shell $N$.   Then  $N\git G$  is an orbifold if and only if
 $N\git G^0$ is an orbifold.
 \end{corollary}

 \section{Quotienting by slice representations of type O}

We establish  Theorem  \ref{mainthm:6} in the case that $G$ is a torus.

Let $T=(\C^\times)^k$ be the diagonal maximal torus of $\GL_k(\C)$ acting  on $W=\C^k$, $k\geq 2$. For $t\in T$ let $t_i(t)$ denote its $i$th component, $1\leq i\leq k$.  Let $x_1,\dots,x_k$ be the usual coordinate functions on $W$ and let $\xi_1,\dots,\xi_k$ be the dual coordinate functions on $W^*$.  Let $D\subset T$ denote the nonzero scalar multiples of the identity matrix. Let $\vec x$ and $\vec \xi$ denote $(x_1,\dots,x_k)$ and $(\xi_1,\dots,\xi_k)$, respectively.  Let $H$ be a torus of dimension $k-1$ and let $\rho\colon H\to T\subset \GL(W)$ be a stable faithful representation. Then $\chi_i:=t_i\circ\rho$ is a character of $H$, $1\leq i\leq k$. Now
$\C[W]^H\simeq\C[f]$  where $f(\vec x)=\prod x_i^{m_i}$ is an invariant monomial and $m_i\in\N$, $i=1,\dots,k$. Since $f$ is not a power of another element in $\C[W]^H$,  the GCD of the $m_i$ is $1$. Set   $m=\sum_{i=1}^k m_i$ and $\vec m=(m_1,\dots,m_k)$.
Let $F:=\rho\inv(D)\subset H$. Then $F\simeq\rho(F)\simeq\Z/m\Z$.

\begin{lemma} \label{lem:H-invariants}
The following hold.
\begin{enumerate}
\item $\rho(H)$ is the subtorus of $T$ defined by $\prod_it_i^{m_i}=1$.
\item The $H$-invariants of $W\oplus W^*$ are generated by $f(\vec x)$, $f(\vec \xi)$ and $x_i\xi_i$, $i=1,\dots,k$.
\end{enumerate}
\end{lemma}

\begin{proof}
Since the GCD of the $m_i$ is $1$, $\vec m$ is a basis element of $\Z^k$, hence the equation in (1) does indeed define a subtorus of $T$.

Let $\vec c=(c_1,\dots,c_k)\in\Z^k$. Since $\rho(H)\subset T$ is a subtorus of rank $k-1$, the relations of the $t_i$ on $\rho(H)$ form a free $\Z$-module of rank $1$. Thus
\begin{equation}\tag{$*$}
\prod_i\chi_i^{c_i}=1   \text{ on } H \text{ if and only if } \vec c=a\cdot\vec m,  \, a\in\Z.
\end{equation}
If (2) fails, then there is an invariant of the form $\prod_{i\in I}x_i^{a_i}\prod_{j\in J}\xi_j^{b_j}$ where $I$, $J\subset\{1,\dots,k\}$ are disjoint nonempty subsets and the $a_i$ and $b_j$ are strictly positive. It follows that there is a $\vec c$ violating $(*)$.
\end{proof}

\begin{remark}
 Let $\lie t$ (resp.\ $\lie h$) denote the Lie algebra of $T$ (resp.\ $H$). If $A\in\lie t$, then $A(\vec x)=(a_1x_1,\dots,a_kx_k)$ for some $a_i\in\C$. We say that $(a_1,\dots,a_k)$ is \emph{associated to $A$\/}. Let $\rho_*\colon \lie h\to \lie t$ be the homomorphism induced by $\rho$. Then $A\in\rho_*(\lie h)$ if and only if $\sum m_ia_i=0$ where $\vec a$ is associated to $A$.
\end{remark}

Let $N_H$ denote the shell of $W$  and set $M_H:=N_H\git H$.

\begin{corollary}
\begin{enumerate}
\item On $N_H$ we have $x_i\xi_i/m_i=x_j\xi_j/m_j$ for  $1\leq i,j\leq k$.
\item The invariants of $N_H$ are generated by $ f(\vec x)$, $ f(\vec \xi)$ and
$$
h(\vec x,\vec \xi)= \sum_{i=1}^k x_i\xi_i.
$$
\end{enumerate}
\end{corollary}

\begin{proof}
Let  $\vec a\in\Z^k$   where $a_i=-m_j$, $a_j=m_i$ and the remaining entries of $\vec a$ are zero. Then $\vec a$ corresponds to an element of $\rho(\lie h)$.
Hence, on $N_H$, $-m_jx_i\xi_i+m_ix_j\xi_j=0$ and we have (1). Part (2) follows easily.
\end{proof}

  Let $W_D=\C$ with $\rho(F)$ and $D$ acting  by scalar multiplication. Let $x_0$ denote the usual coordinate function on $W_D$ with dual coordinate function $\xi_0$ on $W_D^*$.   We have a  $D$-equivariant embedding $\rho\colon W_D\to W$ sending the generator $1\in W_D$ to $v_0:=(\sqrt{m_1},\dots,\sqrt{m_k})\in W$.
  Let $\rho^*\colon W_D^*\to W^*$ be the dual embedding with image $\C\cdot v_0^*$. Let $N_0=W_D\oplus W_D^*$ and let $\phi\colon N_0\to W\oplus W^*$ denote the $D$-equivariant embedding induced by $\rho$ and $\rho^*$.

\begin{corollary} \label{cor:reduce.N_0}
Let $c:=\prod_i m_i^{-m_i/2}$ and let $f_1=c\cdot f(\vec x)$, $f_2=c\cdot  f(\vec\xi)$ and let $h(\vec x,\vec\xi)$ be as above.
Then the embedding $\phi$ has image in $N_H$ and induces a $D$-equivariant isomorphism $N_0/F\simeq N_H\git H=M_H$. In fact,  $\phi^*f_1=x_0^m$, $\phi^*f_2=\xi_0^m$ and $\phi^*h = m\cdot x_0\xi_0$.
Moreover, $\phi$ is a symplectomorphism up to rescaling the standard symplectic form on $W_D\oplus W_D^*$.
\end{corollary}

\begin{proof}
Let $A\in\lie h$ correspond to $\vec a\in\C^k$. Then for $a$, $b\in \C$ we have
$$
<a\cdot v_0^*,A(b\cdot v_0)>={ab} \sum_{i=1}^km_ia_i =0.
$$
Thus $\phi(N_0)\subset N$.  Now
$$
 \phi^*f_1=c\cdot \prod_{i=1}^k\left(\sqrt{m_i} x_0\right)^{m_i}=c\cdot \prod_{i=1}^km_i^{m_i/2}x_0^m= x_0^m
$$
and similarly, $\phi^*f_2=\xi_0^m$. Finally,
$$
\phi^*h(\vec x,\vec\xi)=(m_1+\dots+m_k) \cdot x_0\xi_0=m\cdot x_0\xi_0.
$$
A simple computation shows that $\phi$ is a symplectomorphism up to rescaling.
 \end{proof}

\begin{example}
\label{ex:-ab}
Suppose that $H=\C^\times$ acts on $W=\C^2$ with weight matrix $(-a,b)$ where $a,b > 0$ and $\gcd(a,b) = 1$. Then $W$ is a stable faithful $G$-module. We have $\rho(s) = (s^{-a}, s^b)$ for $s\in H$, $\vec{m} = (b,a)$ and $F$ consists of the $m= (a+b)$th roots of unity in $H$. Then $M_H$ is isomorphic to the orbifold $N_0/F$.
\end{example}

The following is a generalization of \cite[Theorem 1]{HerbigLawlerSeaton20} and gives Theorem \ref{mainthm:6} in the case that $G$ is a torus.

\begin{theorem}\label{thm:main2}
Let $G$ be a torus with $\dim G = \ell >0$, and let $V$ be a
 faithful $1$-modular $G$-module with shell $N$. There is a subgroup $G'\subset G$ and a
faithful  $1$-modular $G'$-submodule $V'\subset V$
with shell $N'$,  satisfying the following.
\begin{enumerate}
\item There is a $G'$-equivariant injection  $N'\to N$ inducing
a Poisson
isomorphism  $N'\git G'\tosim N\git G$.
\item $N'$ is  minimal.
\item $G'$ is the subgroup of $G$ which stabilizes $V'$.
\item $(V',G')$  is a direct sum of modules $(V'',G'')$ with shells $N''$ of two kinds.
\begin{enumerate}
\item $G''$ is a finite cyclic subgroup of $\C^\times$ acting by scalar multiplication on $V'':=\C$, or
\item $G''$ is a torus,   $V''$ is a
 faithful  $1$-modular $G''$-module and   $N''$ is minimal.
\end{enumerate}
\end{enumerate}
At most one summand of the second type occurs. Moreover, if $V$ is a stable $G$-module, then $V'$ is a stable $G'$-module.
\end{theorem}

\begin{proof}
We give a recipe for constructing the $(V'',G'')$ as in
(4)
satisfying
(a) and (b).  We know that, UTCLS, $V$ is stable, so we may assume that in the following.

Let  $(S,H)$ be a symplectic slice representation of $N$ of type O.  We may decompose $V$ as    $V^H\oplus W$  where $V^H$ and $W$ are  $G$-submodules of $V$. Clearly $W$ is a stable faithful $H$-module and $\dim W\git H=1$.
 We use some of the notation  and  results above. Let $\rho\colon G\to  T\subset\GL(W)$ give the $G$-module structure on $W$ and let $K=\Ker\rho$. Then $V^H$ is a faithful stable $K$-module.   Let $N_K$ denote the shell corresponding to the $K$-action on $V^H$.
Let $\chi_i=t_i\circ \rho$, $i=1,\dots,k$, considered as characters of $H$.
There is an  $\vec m\in\N^k$, where $m_1,\dots,m_k$ have GCD $1$, such that
$$
\rho(H^0)= \{(t_1,\dots,t_k)\in T \mid \prod_it_i^{m_i}=1\}.
$$
Let $m=\sum_i m_i$ and let
$N_0=W_D\oplus W_D^*$ where $W_D$ is as described before Corollary   \ref{cor:reduce.N_0}.
Since $G$ is connected, $\rho(G)=\rho(H^0)$ or $\rho(G)=T$.
\smallskip

Case 1. Suppose that $\rho(G)=\rho(H^0)$. Then $H=H^0$, $K\simeq G/H$ and $G\simeq H\times K$.   Thus $N= N_H\times N_K$ and $N\git G\simeq N_H\git H\times N_K\git K$.
Recall that $D$ denotes the scalar multiples of the identity in $T$ and
let $F=\rho\inv(D \cap \rho(H))$. Then $F\simeq\rho(F)\simeq  \Z/m\Z$. The inclusion $N_0 \to N_H $   induces
a Poisson isomorphism.
$$
(N_0\times N_K)\git (F\times K)\tosim N_H\git H\times N_K\git K\simeq N\git G.
$$
We are in case (4a) above with $G''=F$ and $V''=W_D$. Note that the subgroup of $H$ stabilizing $W_D\subset W$ is $G''$. We continue  by examining the   action of $K$ on   $N_K=N_{V^H}$ where $\dim K<\dim G$ and $V^H$ is a stable faithful $K$-module.

\smallskip

Case 2. Suppose that $\rho(G)=T$.   Let $R$ denote a copy of $\C^\times$ in $G$ such that $\rho(R)= D$. We have finite groups $F_H= R\cap H$ and  $F_K=R\cap K$.
Since $H\cap K=\{e\}$,
$$
F_H\simeq \rho(F_H)\simeq \rho(H)\cap D\simeq\Z/m'\Z\text { where }m'=m[H:H^0].
$$
Let  $n=\dim V$ and
let $x_{k+1},\dots,x_n$ be coordinate functions on $V^H$ which transform by characters of $G$. Let $\xi_{k+1},\dots,\xi_n$ be the dual variables.    Let $\vec a\in\Z^n$ denote the weights of the $R$-action on $W\oplus V^H$. Then $\vec a$ is associated to a generator $A$ of $\lie r$ with associated quadratic function $f_A$ and zero set $Z$. We have
$$
 f_A(\vec x,\vec \xi)=\sum_{i=1}^na_i x_i\xi_i\text{ and }N=(N_H\times N_K)\cap Z.
$$
Now $G$ contains the subgroup $H\times K$ and $R\cap(H\times K)=F_H\times F_K$.
Since $RH^0K^0$ is a subtorus of $G$ of dimension $\ell$,  $RH^0K^0=RHK=G$ so that $RHK$ is connected, $H=F_HH^0$ and $K=F_KK^0$.
Thus  $RH$ and $RK$ are connected. Let $G''=RK\subset G$ and let $V''=W_D\oplus V^H$. Then $G''$ is a subtorus of $G$ and stabilizes $V''$.
Since $R$ acts as scalars on $W$ and $W_D$, the fact that $G''$ acts faithfully on $W\oplus V^H$ implies that $G''$ acts faithfully on $V''$,
and the fact that $W\oplus V^H$ is a stable $G''$-module implies that $V''$ is a stable $G''$-module as well.
 Let $g\in G$ and suppose that $g$ stabilizes $V''$. Then $g$ has to act on $w_0\in W_D\subset W$ via scalars. Thus $g\in RK=G''$. If $g$ acts as the identity, then since $K$ acts faithfully on $V^H$ and $R\to\GL(W_D)$ has kernel $F_K$, $g$ is the identity of $G''\simeq (R\times K)/F_K$. Thus $G''$ acts faithfully on $V''$.

Let  $N''$ denote the shell of $V''$ for the $G''$-action.
Then $N''=(N_0\times N_K)\cap Z$. The inclusion $N_0 \to N_H$  induces a $G''$-equivariant inclusion $N''\to N$ inducing
a Poisson  isomorphism
$$
N''\git G''\simeq N\git G.
$$
(We no longer need to divide by
the subgroup $F$ from Case 1 since it is a subgroup of $R$.)\
We have $\dim G''<\dim G$ and (4b)
holds except that $N''$ may not be minimal. If not, then there is a symplectic slice of type O and we can continue.
\end{proof}

We illustrate the application of Theorem \ref{thm:main2} and the recipe indicated in its proof with the following examples. Given a weight matrix
$A\in\Z^{\ell\times n}$, one may first use Corollaries \ref{cor:maximal} and \ref{cor:matrixA} to identify slices of type O. The first simple example
illustrates the process in Case 1 of the proof of Theorem \ref{thm:main2}. Throughout, we denote by $e_i$ the elements of the standard basis of $\C^n$.

\begin{example}
Suppose that $G=(\C^\times)^2$ acts on $V=\C^4$ with weight matrix $\begin{pmatrix} -1&1&0&0\\0&0&-1&1\end{pmatrix}$. Then $V$ is a stable faithful
$G$-module and $N\git G$ is obviously a product. We have two symplectic slice representations of type O where $W$ is the span of $e_1, e_2$ and $e_3, e_4$, respectively, and in both cases $\rho(G) = \rho(H^0)$.
Applying Example \ref{ex:-ab} with $a=b=1$ to each yields an isomorphism
\[
    (\C^2 \times \C^2)\git (\{\pm 1\}\times \{\pm 1\})\tosim N\git G.
\]
\end{example}

The next example illustrates Case 2 of the proof of Theorem \ref{thm:main2}.

\begin{example}
Suppose that $G=(\C^\times)^2$ acts on   $V=\C^4$ with weight matrix $\begin{pmatrix} -1&0&2&2\\0&-2&5&5\end{pmatrix}.$  Then $V$ is a faithful stable $G$-module.
We have a symplectic slice representation of type O where $W$ is the span of $e_1,e_2$,
$$
H=\{(s_1,s_2)\in G\mid s_1^2s_2^5=1\}
$$
and $V^H$ is the span of $e_3,e_4$.
The map $\rho$ is given by $\rho(s_1,s_2) = (s_1^{-1}, s_2^{-2})$ so that
$\rho(H)=\{(t_1,t_2)\in T\mid t_1^4t_2^5=1\}$
and $K = \{(1,\pm 1)\}$. Thus $\vec m=(4,5)$, $m = 9$ and $\rho(G) = T$ so that we are in Case 2 of the proof of Theorem \ref{thm:main2}.
We may take $R = \{(t^2,t)\mid t\in\C^\times\}$ and then $F_H = \{(\zeta^2,\zeta)\mid\zeta^9 = 1\}\simeq\Z/9\Z$ and
$F_K = K\simeq\Z/2\Z$.
Now
$R$ acts on $V^H$ with weight $9$ and on $W$ by weight $-2$. Replacing $W$ with $W_D\simeq\C$, we still have an action of $R$ with weight $-2$.
The $F_H$-action on $V'=W_D\oplus V^H$ is as a subgroup of $R$.  Thus the theorem allows us to replace $(V,G)$ by
$(V',G')$  where $V' \simeq \C^3$ and
$G'=R\simeq \C^\times$ acts with weight matrix $(-2, 9, 9)$.
Note that
$V'$ is a minimal $G'$-module.
\end{example}

The next example is similar to the one above, except that   $H$ is not connected
and   the results of \cite{HerbigLawlerSeaton20} do not apply.

\begin{example}
Let $G=(\C^\times)^2$ act on $V=\C^4$ with weight matrix $\begin{pmatrix} 3&0&-4&6\\1&-3&0&0\end{pmatrix}.$ We  have a symplectic slice representation of type O with $H = \{ (\pm 1,s)\mid s\in\C^\times\}$ where $W$ is the span of $e_1$, $e_2$  and $ V^H$ is the span of $e_3$, $e_4$. Then
$$
\rho(H^0)=\{(t_1,t_2)\in T\mid t_1^3t_2=1\}
$$
 so that $\vec m=(3,1)$ and $m=4$. We may take $G'=R = \{(t^4, t^{-3}) \mid t\in\C^\times\}$ which acts on $W$ and $W_D$ with weight $9$ and on $V^H$ with weights $-16$ and $24$. Let  $V'=W_D\oplus V^H$.  Then $V'\simeq\C^3$ is a minimal $G'$-module   and $N'\git G'\simeq N\git G$.
\end{example}

\section{Lifting Isomorphisms}\label{sec:lifting}

Let $X$ be an affine $G$-variety where $G$ is reductive. Let $\sigma$ be an (algebraic) automorphism of $G$ and let    $\phi\colon X\git G\to X\git G$ and  $\Phi\colon X\to X$ be   morphisms. We say that $\Phi$ is a \emph{lift of $\phi$\/} if $\pi\circ\Phi=\phi\circ\pi$ and we say that $\Phi$ is \emph{$\sigma$-equivariant\/} if $\Phi\circ g=\sigma(g)\circ \Phi$ for all  $g\in G$.    In this section we take as the  strata of $X$  the \emph{irreducible components\/} of the   $(X\git G)_{(H)}$ for $H$ reductive. We say that $\phi$ is \emph{strata preserving\/} if it permutes the strata of $X\git G$.

Let $\A(X)$ (resp.\ $\A(X\git G)$) denote the algebraic derivations of $\C[X]$ (resp.\ $\C[X]^G$).
  Since any element of $\A(X)^G$ acts as a derivation of $\C[X]^G$, we have a natural morphism $\pi_*\colon \A(X)^G\to \A(X\git G)$. If   $B\in\A(X)^G$ and $\pi_*B=A$, we say that \emph{$B$ is a lift of $A$}. For $x\in X$ let $T_xX$ denote the Zariski tangent space at $x$. If $A\in\A(X)$, then we  naturally have a vector $A(x)\in T_xX$ so one can think of elements of $\A(X)$ as vector fields on $X$.

  The following  will be useful \cite[Corollary 6.6]{GWSlifting}.

\begin{lemma}\label{lem:same.dim.lifting}
Let $G$ be a reductive complex group and $X$ a normal affine $G$-variety. Suppose that $X$ is $2$-principal and that all $G$-orbits have the same dimension (in which case all $G$-orbits are closed).
Then $\pi_*\A(X)^G\to\A(X\git G)$ is surjective.
\end{lemma}

 Let $G$ be quasi-toral and $V$ a faithful $1$-modular
 minimal
 $G$-module with shell $N$.
 We want to apply some results of \cite{SchVectorFields} and \cite[\S 5]{GWSQuotients}  to the $G$-action on $N$. The results there are stated for  $2$-principal $G$-modules or quotients of such $G$-modules by semisimple groups so the arguments have to be modified.

Let $\pi\colon N\to M$ be the quotient mapping.
Let $M_0$ denote the complement of the strata of $M$ whose corresponding isotropy groups are
 infinite
 and let $N_0=\pi\inv(M_0)$.

\begin{proposition}\label{prop:good}
Let $G$, etc.\ be as above. Then the following hold.
\begin{enumerate}
\item $N$ is stable and  $2$-principal with  rational singularities.
\item $N_0$ is smooth.
\item $\codim_M(M\setminus M_0)\geq 4$.
\item The strata of $M$ are smooth.
\item $N$ is factorial, i.e., $\C[N]$ is a UFD.
\item $\C[N]^*=\C^\times$.
\end{enumerate}
\end{proposition}

\begin{proof}
By Remark \ref{rem:rational-sings} $N$ is stable with rational singularities
 and it is  $2$-principal by Lemma \ref{lem:codim.Npr},  giving (1).
 The open set $N_0$ consists of closed orbits with finite isotropy groups, hence lies in $N_\sm$ giving (2). The strata of $M$ are even dimensional  and   $N$ has no symplectic slice representations of type O, hence (3) holds.  The symplectic slice theorem gives (4).  The proof of \cite[Cor.\ 4.5]{HerbigSchwarzSeaton3} gives (5) and  \cite[Corollary 4.5(2)]{HerbigSchwarzSeaton3} gives (6).
\end{proof}

\begin{lemma}\label{lem:pi_*.surjective}
Let $G$, etc.\ be as above. Then $\pi_*\colon \A(N)^G\to\A(M)$ is surjective.
\end{lemma}

\begin{proof}
First consider the case that $G$ is a torus.
Let $A\in\A(M)$.   Let $\{U_\alpha\}$ be an open cover of $M_0$ by affine open sets and let $N_\alpha=\pi\inv(U_\alpha)$ for all $\alpha$. Then the $U_\alpha$ are smooth  and all the $G$-orbits on $N_\alpha$ have dimension $\ell=\dim G$.  Since
 $N$
 is $2$-principal,   Lemma \ref{lem:same.dim.lifting} shows that there are $B_\alpha\in\A(N_\alpha)^G$ such that $\pi_*B_\alpha=A|_{U_\alpha}$. The differences $B_\alpha-B_\beta$ are $G$-invariant vector fields  which are tangent to the $G$-orbits.   Such vector fields on $N_0$ are the free $\C[M_0]$-module generated by $\lieg$, thought of as vector fields on $N_0$. Thus the $\{B_\alpha-B_\beta\}$ give us an element of $H^1(M_0,\O_M^\ell)$ where $\O_M$ denotes the sheaf of regular functions on $M$. By \cite{Boutot}, $M$ has rational singularities, hence is Cohen-Macaulay. Then Proposition  \ref{prop:good}(3) implies that the class $\{B_\alpha-B_\beta\}$ is trivial \cite[Proposition 10.4]{GWSliftingHomotopies} and hence there are elements $C_\alpha\in
 \C[N_\alpha]^G\otimes \lieg$
 such that $C_\alpha-C_\beta=B_\alpha-B_\beta$ for all $\alpha$, $\beta$. Let $B_\alpha'=B_\alpha-C_\alpha$
 for all $\alpha$.
 Then  $B_\alpha'$ and $B_\beta'$ agree on
 $N_\alpha\cap N_\beta$
 and we obtain an element
 $B_0\in \A(N_0)^G$
 such that $\pi_*B=A|_{M_0}$. Since $N$ is normal and $N_\pr\subset N_0$, $\C[N]=\C[N_0]$. Since $B_0$ is a $G$-invariant derivation of $\C[N_0]=\C[N]$, we may consider $B_0$ as an element $B\in \A(N)^G$. Since $A$ and $B$ have the same action on $\C[N]^G$,   $\pi_*B=A$.

Now we drop the assumption that $G$ is a torus. Let $M^0=N\git G^0$ so that $M=M^0/F$ where $F=G/G^0$.
Since $N$ is normal  and $2$-principal,  $M^0$ is normal and $2$-principal for the action of $F$.
By Lemma \ref{lem:same.dim.lifting}, any $A\in\A(M)$ lifts to $B\in \A(M^0)^F$ and by the paragraph above, $B$ lifts to $C\in\A(N)^{G^0}$ and averaging over $F$ we
can
arrange that $C\in\A(N)^G$.
\end{proof}

\begin{lemma}\label{lem:tan.spaces} Let $G$, etc.\ be as above.
Let $Q=M_{(H)}$ and let $q\in Q$.   Then
\begin{equation}\tag{$*$}
\{A(q)\mid A\in\pi_*\A(N)^G\}=T_qQ.
\end{equation}
\end{lemma}

\begin{proof}  Let $Gx$ be a closed orbit in $N$ with $G_x=H$ and   $\pi(x)=q$.  Let $(S,H)$ be the symplectic slice representation at $x$ and decompose $N_S=S^H\times N_0$ as usual.  By the symplectic slice theorem, there is an  $H$-saturated affine neighborhood  $U$ of $0\in N_S$ and   an excellent morphism (see \cite[3.10--3.15]{HerbigSchwarzSeaton2})
$$
\rho\colon X=(G\times^HU)\to Y\subset N,\ \rho([e,0])= x,
$$
where  $Y$ is a $G$-saturated  affine  neighborhood of $x$ in $N$. By  \cite[Proof of Corollary 4.4]{GWSlifting}   we   obtain a commutative diagram
$$
\begin{CD}
\C[X]^G \otimes_{\C[Y]^G} \A(Y)^G @>\sim>> \A(X)^G   \\
 @VV \id \otimes (\pi_Y)_*V  @VV(\pi_X)_*V   \\
\C[X]^G \otimes_{\C[Y]^G} \A(Y\git G) @>\sim>>\A(X\git G).
\end{CD}
$$
Now $X$ is an affine open subset of $\widetilde X=G\times^HN_S\simeq (G\times^HN_0)\times S^H$ and we have a direct sum decomposition $\A(\widetilde X)^G=\Sigma_1\oplus \Sigma_2$ where
$$
\Sigma_1=\A(G\times^HN_0)^G\otimes\C[S^H]\text { and }\Sigma_2=\C[G\times^HN_0]^G\otimes\A(S^H).
$$
Note that $(\widetilde X\git G)_{(H)}\simeq S^H$.
Now $T_{[e,0]}(G\times^HN_0)=T_x(Gx)\oplus (S_0=T_0(N_0))$ and the value of $A\in \A(G\times^HN_0)^G$ at $[e,0]$ is in $T_x(Gx)^H\oplus S_0^H$.  Since the quotient mapping is constant on $Gx$, $T_x(Gx)^H$ is sent to $0$ and $S_0^H=(0)$. If $A\in\A(S^H)$, then $A(0)$  is an arbitrary element of $S^H$. This shows that $(*)$ holds for $M$ replaced by   $\widetilde X\git G$ or $X\git G$. Let $M'=X\git G$ and $M''=Y\git G$. Since $\rho$ is excellent, it induces an \'etale  morphism $\bar\rho\colon M'\to M''$ which restricts to an \'etale morphism   $M'_{(H)}\to M''_{(H)}$ and an isomorphism $T_q M'_{(H)}\to T_qM''_{(H)}$.  Since $(*)$ holds for $M'$, it holds for $M''$. Since $M''$ is an affine neighborhood of $q\in M$, $(*)$ holds for $M$.
\end{proof}

By the argument of \cite[Cor.\ 2.3]{SchVectorFields}, Lemmas \ref{lem:pi_*.surjective} and \ref{lem:tan.spaces} establish the following.

\begin{corollary}\label{cor:permute.comps}
Let $\phi$ be an algebraic automorphism of $M$ and let $Q$ be a stratum of $M$ of dimension $k$. Then $\phi(Q)$ is a stratum of $M$ of dimension $k$.
\end{corollary}

\begin{corollary}\label{cor:lifting}
Let $G$, etc.\ be as above and let $\phi$ be an algebraic  automorphism of $M=N\git G$. Then there is an algebraic automorphism $\Phi$ of  $N$ which lifts $\phi$.
\end{corollary}

\begin{proof}
Let $M^0=N\git G^0$ so that $M=M^0/F$ where $F=G/G^0$. Then the action of $F$ on $M^0$ is $2$-principal and by \cite[Prop.\ 5.1]{GWSQuotients}, $\phi$ lifts to an automorphism of $M^0$, so we may reduce to proving that automorphisms of $M^0$ lift to $N$. The proof of \cite[Prop.\ 5.2]{GWSQuotients}   establishes lifting of automorphisms of $M^0$ to automorphisms of $N$ provided that $N$ is factorial and $2$-principal which is indeed the case by Proposition \ref{prop:good}.
\end{proof}

\begin{proposition}\label{prop:lifting}
Let $G$, etc.\ be as above.
Let $\phi$ be an algebraic  automorphism of $M$ and let $\Phi$ be an automorphism  lifting   $\phi$. Then $\Phi$ is $\sigma$-equivariant for an automorphism $\sigma$ of $G$.
\end{proposition}

\begin{proof}
 We follow the proof of  \cite[Theorem 1.12, p.\ 180]{GWSQuotients}. Since $\phi$ induces an isomorphism of $N^G=(V\oplus V^*)^G$ with itself, we may assume that $\phi(0)=0$.   Since $N$ is a cone, we have an action of $\C^\times$ by scalar multiplication and this descends to an action on
 $M$.
 Let
 $\phi_t(y)=t\inv\cdot\phi(t\cdot y)$, $y\in M$, $t\in\C^\times$.
 As in   \cite[Proposition 2.9]{GWSQuotients}, $\Phi(0)=0$  and $\phi_0=\lim_{t\to 0}\phi_t$ exists and is an automorphism of $M$ which commutes with the action of $\C^\times$.  Now $T_0N=V\oplus V^*$   and    $\phi_0$ lifts to $\Phi'(0)\in \GL(V\oplus V^*)$ where $\Phi'(0)$ normalizes $G$.  Our automorphism  $\sigma$ of $G$ is conjugation by $\Phi'(0)$.

Set $\Phi_t=t\inv\circ\Phi\circ t$, $t\in \C^\times$, and set $\Phi_0=\Phi'(0)$. Set $\Psi_t:=\Phi_t\circ\Phi_0\inv$ and
 let $x\in N_\pr$. Then $\Psi_t(gx)=\gamma(t,x,g)\cdot \Psi_t(x)$ where $\gamma\colon\C\times N_\pr\times G\to G$ is a morphism. Since $N$ is $2$-principal and normal, $\gamma$ extends to a morphism $\C\times N\times G\to G$. Now $\gamma(0,x,g)=g$ for all  $x$ and $g$ so that  $g\inv\gamma(t,x,g)$ always lies in $G^0$. For $g$ fixed, $(t,x)\mapsto g\inv\gamma(t,x,g)$ is  a morphism $\C\times N\to G^0$. Since $\C[\C\times N]^*=\C[N]^*=\C^\times$, the morphism $g\inv\gamma(t,x,g)$ is independent of $t$ and $x$, hence equals $g\inv\gamma(0,x,g)=e$.  It follows that $\gamma(t,x,g)=g$ for all $t$, $x$ and $g$. Hence all the $\Psi_t$ are $G$-equivariant and our original $\Phi$ is $\sigma$-equivariant.
\end{proof}

The following is  Theorem \ref{mainthm:3}.

\begin{theorem}\label{thm:isomorphism}
For $i=1$, $2$, let $G_i$ be quasi-toral and $V_i$ a faithful $1$-modular
minimal
$G_i$-module with shell $N_i$.
 If
$$
N_1\git G_1\simeq N_2\git G_2
$$
 as affine varieties, then there is a linear isomorphism
$$
\Gamma\colon V_1\oplus V_1^*\tosim V_2\oplus V_2^*
$$
inducing  isomorphisms of $N_1$ and $N_2$ and of $G_1$ and $G_2$.
\end{theorem}

\begin{proof}
We follow the proof of \cite[Corollary 1.14]{GWSQuotients}. Let $M_i=N_i\git G_i$, $i=1$, $2$.
Let $N=N_1\times N_2$ , $M=M_1\times M_2$ and $G=G_1\times G_2$. Let $\psi\colon M_1\to M_2$ be the isomorphism. Clearly, $\psi$ sends
$$
N_1^{G_1}\simeq V_1^{G_1}\oplus (V_1^{G_1})^* \text{ isomorphically onto }  N_2^{G_2} \simeq  V_2^{G_2}\oplus (V_2^{G_2})^*.
$$
We may thus assume that $\psi(0)=0$.
Let $(m_1,m_2)\in M$ and define   $\phi(m_1,m_2)=(\psi\inv(m_2),\psi(m_1))$. Then $\phi$ is an automorphism of $M$ of order $2$ and fixes the image of $0$. By Corollary \ref{cor:lifting} and Proposition \ref{prop:lifting} $\phi$ has a $\sigma$-equivariant lift $\Phi\colon N\to N$. By \cite[Lemma 19]{Kuttler}, $\Phi$ sends closed orbits to closed orbits, hence $\Phi(0)=0$.  Since $\Phi$ lies over $\phi$ and sends closed orbits to closed orbits, it sends $(N_1)_\pr\times \{0\}$ isomorphically onto $\{0\}\times (N_2)_\pr$, hence it interchanges $N_1\times\{0\}$ and $\{0\}\times N_2$.
Hence
$\Phi'(0)$ restricts to a linear isomorphism
$$
\Gamma\colon V_1\oplus V_1^*\to V_2\oplus V_2^*.
$$
Since $\Phi$ and $\Phi'(0)$ are $\sigma$-equivariant, $\sigma$ interchanges $G_1$ and $G_2$. Hence $\Gamma$ induces an isomorphism of $G_1$ and $G_2$ and sends $N_1$ onto $N_2$.
\end{proof}

The following lemma is needed to deduce Theorem \ref{mainthm:6} from Theorem \ref{thm:main2}.

\begin{lemma}\label{lem:lifting}
Let $G$, etc.\ be as usual with $G$ connected, i.e., a torus. Suppose that a finite group $\Gamma$ acts effectively and algebraically on $M$. Then there is a reductive group $\tilde G$  which is an extension of $G$ by $\Gamma$ which acts on $N$ such that the action of
  $\tilde G/G$ on $M$ is the given action of $\Gamma$.
\end{lemma}

\begin{proof}
By Proposition \ref{prop:lifting} every $\gamma\in\Gamma$ has a (non unique) lift $\tilde\gamma\in\Aut(N)$ which normalizes $G$. If $\tilde\gamma'$ is also a lift of $\gamma$, then $\tilde\gamma'\circ\tilde\gamma\inv$ is a lift of the identity. As in the proof of Proposition \ref{prop:lifting}, one shows that $\tilde\gamma'\circ\tilde\gamma\inv\in G$. If $\gamma_1$, $\gamma_2\in\Gamma$, then $\tilde\gamma_1\tilde\gamma_2$ is a lift of $\gamma_1\gamma_2$ and one easily sees that
$$
\widetilde G:=\bigcup_{\gamma\in\Gamma}\tilde\gamma G
$$
is a quasi-toral reductive group acting on $N$ such that $\widetilde G/G$ acts on $M$ in the same way as $\Gamma$.
\end{proof}

The following is  Theorem \ref{mainthm:6}.
\begin{theorem}\label{thm:main3}
Let $G$ be quasi-toral and $V$ a $1$-modular faithful $G$-module with shell $N$. There is a linear subspace $V'\subset V$ with the following properties.
\begin{enumerate}
\item Let  $G'\subset G$ be the stabilizer of $V'$. Then $V'$ is a $1$-modular
 faithful $G'$-module.
\item There is a $G'$-equivariant  inclusion $N':=N_{V'}\to N$ inducing
a Poisson
isomorphism $N'\git G'\simeq N\git G$.
\item
$N'$ is minimal.
\end{enumerate}
If $V$ is a stable $G$-module, then $V'$ is a stable $G'$-module.
\end{theorem}

\begin{proof}
By Theorem \ref{thm:main2} we have the following. There is a subgroup $G''\subset G^0$ which is the stabilizer of a linear subspace $V'\subset V$.
 Moreover, $V'$ is a $1$-modular faithful $G''$-module. If $V$ is a stable $G$-module, then $V'$ is a stable $G''$-module.
 Let $N'=N_{V'}$. There is a $G''$-equivariant  injection $N'\to N$ inducing
 a Poisson isomorphism $N'\git G''\simeq N\git G^0$ where
$N'$ is minimal.
Now $\Gamma=G/G^0$ acts algebraically and effectively on $N\git G^0\simeq N'\git G''$, and by Lemma \ref{lem:lifting} the action of $\Gamma$ on $N'\git G''$ comes from the action of a reductive group $\widetilde G\supset G''$ where $\widetilde G/G''\simeq\Gamma$.
Let $n\in N'_\pr$ (relative to the action of $G''$) and let $\tilde g\in  \widetilde G$. There is a $g\in G$ such that $g\inv \tilde gn\in G^0n$. Thus we have a morphism $\rho\colon N'_\pr\to G^0$ such that $g\inv \tilde gn=\rho(n)n$, $n\in N'_\pr$. As before, $\rho$ extends to a morphism $N'\to G^0$ which is constant since $\C[N']^*=\C^\times$. Thus there is a $g_0\in G^0$ such that $g_0\inv g\inv \tilde g$ is the identity on $N'$. Hence $\tilde g$ is the restriction to $N'$ of $gg_0\in G$.  If $\tilde g$ is the restriction of $g_1$ and $g_2\in G$ to $N'$, then $h=g_1\inv g_2\in G^0$ and $h$ preserves $N'$, hence preserves $N'\cap (V\times\{0\})=V'$. Thus $h\in G''$.  It follows that  there is a
finite extension $G'$ of $G''$ with $G''\subset G'\subset G$ such that $G'$ stabilizes $V'$ and induces $\Gamma$ acting on $N'\git G''$.
If $g\in G$ and $gV'=V'$, then there is a $g'\in G'$ which has the same image as $g$ in $\Aut(N'\git G'')$. Then $g\inv g'\in G^0$ preserves $N'$, hence $V'$, and by Theorem
\ref{thm:main2},
$g\inv g'\in G''$. Hence $g\in G'$.
Note that $V'$ is a $1$-modular and faithful $G'$-module.
We have established (1)--(3).
If $V$ is stable, then
$(G'', V')$ is stable so that as $G'$ is a finite extension of $G''$, $(G', V')$ is stable as well.
\end{proof}

 \begin{corollary}\label{cor:V'/G'.to.V/G}
 The inclusion $V'\to V$ induces an isomorphism $V'\git G'\tosim V\git G$.
 \end{corollary}

 \begin{proof}
Every closed $G$-orbit in $V\simeq V\times\{0\}\subset N$ intersects $N'$ in a unique closed $G'$-orbit which must lie in $V'\times\{0\}$.  Similarly, every closed $G'$-orbit in $V'\times\{0\}$ is contained in a (unique) closed $G$-orbit in $V\times\{0\}$. Hence $V'\to V$ induces an isomorphism $V'\git G'\simeq V\git G$.
 \end{proof}

\section{Consequences for Real Symplectic Quotients}\label{sec:real}

Let $K$ be a compact Lie group, let $(V,K)$ be a unitary $K$-module and let $G=K_\C$. The $K$-action on $V$ extends to an action of $G$.
Let $\rho\colon V\to\lie k^*$ denote the real homogeneous moment map in terms of the underlying structure of $V$ as a real symplectic vector space.
Let $Z = Z_V = \rho^{-1}(0)$ denote the \emph{real shell},
an algebraic
subset of $V$. Algebraically, $Z$ corresponds to the ideal $I_\rho$ of $\R[V]$ generated by the component functions of $\rho$.
 The \emph{real symplectic quotient} is  $Z/K$.
We say that $K$ is \emph{real quasi-toral}  if $K^0$ is a torus; equivalently, if $G$ is quasi-toral.
In the rest of this section we assume that $K$ is quasi-toral and  we
establish
analogs of Theorems \ref{mainthm:6} and \ref{mainthm:3} for the real shell and symplectic quotient of  faithful
unitary representations of real quasi-toral groups. To do this we will need to assume that $V$ is faithful and $1$-large (equivalently, faithful and stable) as $G$-module.

 \begin{lemma}\label{lem:1-large}
Let $V$   be a faithful $K$-module with real shell $Z$. Let $G=K_\C$.
\begin{enumerate}
\item Every closed $G$-orbit in $V$ intersects $Z$ in a unique $K$-orbit and every point of $Z$ lies on a closed $G$-orbit. The inclusion $Z\to V$ induces a homeomorphism $Z/K\to V\git G$.
\end{enumerate}
Now suppose that $V$ is   $1$-large as  $G$-module.
\begin{enumerate}
\addtocounter{enumi}{1}
\item The ideal $I_\rho$ consists of all the real polynomial functions vanishing on $Z$.  Hence the
\emph{Poisson algebra $\R[Z/K]$ of real regular functions on $Z/K$} is
$\R[V]^K/I_\rho^K$.
\item  The real dimension of $Z$ is the same as the complex dimension of the complex shell $N$.
\end{enumerate}
\end{lemma}

\begin{proof}
The real shell $Z$ is the same thing as the Kempf-Ness set of $V$ \cite[Equations 4.5]{GWSkempfNess}. Then (1)   follows from Theorem 4.2   of loc.\ cit. Parts (2) and (3) follow from  \cite[\S 4]{HerbigSchwarz}.
\end{proof}

If $V$ is not $1$-large, then (2) and (3) fail. However, they become true for the same set $Z$ if we shrink $V$.

\begin{proposition}\label{prop:Kempf-Ness}
Let $K$ be real quasi-toral and $V$ a faithful unitary $K$-module which is not   $1$-large as $G=K_\C$-module. Then there is a unique $G$-submodule  $V'\subset V$ with the following properties.
\begin{enumerate}
\item Let $H=\Ker(G\to\GL(V'))$ and $L=\Ker(K\to\GL(V'))$. Then $V'$ is a stable faithful (hence $1$-large) $G'=G/H$-module and $G'$ is the complexification of $K'=K/L$.
\item The inclusion $V'\to V$ induces an isomorphism $\C[V]^G\tosim\C[V']^{G'}$.
\item $Z\subset V'$ is the real shell for the $K'$-action on $V'$ and $Z/K=Z/K'$ as sets.
\end{enumerate}
\end{proposition}

\begin{proof}
In \cite[Lemma 5.1]{HerbigSchwarzSeaton2} this result is established for $K$ replaced by $K^0$ (and $G$ replaced by $G^0$). Since $V'$ is unique (for the action of $G^0$), it is $G$-stable, so the
results of the proposition hold for $G$ and $K$.
\end{proof}

By the above, we can always reduce to the case that $V$ is $1$-large as $G$-module, so we assume this for the rest of this section.

Let $X$ be a complex affine variety. An anti-holomorphic involution $\omega$ on $X$ is called a \emph{real structure on $X$}. The \emph{real points of $X$} are the fixed points of $\omega$.
We recall the following from \cite[\S 4]{HerbigSchwarz}.
Let $J$ denote the complex structure on $V$ given by multiplication by $i$. Then the complexification $V_\C = V\otimes_\R \C$ decomposes into the
$\pm i$-eigenspaces $V_{\pm}$ of $J$. We have $V_+=\{v-iJv\mid v\in V\}$ and $V_-=\{v+iJv\mid v\in V\}$. Thus $V_\C=V_+\oplus V_-$ where $V_+\simeq V$ by the obvious mapping. Note that $V_-$ is isomorphic to $V$ with the conjugate complex structure given by $-J$. We have a nondegenerate pairing of $V_+$ and $V_-$ which sends $(v-iJv,v'+iJv')$ to $\langle v,v'\rangle$ where $\langle\, ,\rangle$ is the $K$-invariant hermitian form on $V$.
Thus $V_-\simeq V^*$. Let $(v_+,v'_-)\in V_+\oplus V_- = V_\C$ and let $\tau(v_+,v'_-)=(v'_+,v_-)$. The set of fixed points of $\tau$ is the set
$\{(v_+,v_-) \mid v\in V\}$; we denote by $V_\R$ this subspace of $V_\C$ isomorphic to $V$ (as a complex vector space).
It is easy to see that $\tau$ is anti-holomorphic, i.e., $\tau\circ J=-J\circ \tau$,
giving $V_\C\simeq V\oplus V^\ast$  a real structure.

We have real structures on $G$ and $\lieg$,
as follows.
Let $\liek$ denote the Lie algebra of $K$. We can decompose $\lieg$ as $\liek\oplus i\liek$. We have an anti-holomorphic involution $\sigma_*$ on $\lieg$ which sends $A+iB$ to $A-iB$ where $A$, $B\in\liek$. Using the exponential map  we have a decomposition $G=K \exp(i\liek)\simeq K\times\exp(i\liek)$ and the corresponding anti-holomorphic \emph{Cartan involution $\sigma$ of $G$}  sends $k\exp(iA)$ to $k\exp(-iA)$ for $k\in K$, $A\in\liek$. The fixed
points
 of $\sigma$ are $K$.

If  $\phi\colon N_1\to N_2$ is a morphism of two complex varieties $N_1$ and $N_2$ with real structures $\tau_1$ and $\tau_2$, we say that $\phi$ is a \emph{real morphism\/} if
$\phi\circ\tau_1=\tau_2\circ\phi$. We similarly define \emph{real isomorphisms} of complex affine varieties with real structures,
\emph{real\/} linear isomorphisms of complex vector spaces with real structures, and \emph{real\/} isomorphisms of reductive complex Lie groups with real structures (e.g., $K_\C$) to be isomorphisms that preserve the real structures.

The (complex) moment map $\mu\colon V\oplus V^*\to\lieg^*$ equals $\rho\otimes_\R \C$. It follows that $\mu\colon V\oplus V^\ast\to\lieg^*$ is real where $\lieg^*$ has the real structure given by the dual of $\sigma_*$.
Thus $\tau$ preserves the (complex) shell $N=\mu^{-1}(0)$.
The natural map  $Z\to V_\R\subset V_\C$ has image in the real points of $N$.
From \cite[Cor.\ 4.2  and Prop.\ 4.7]{HerbigSchwarz} we have:
\begin{lemma}\label{lem:1-large2}
Let $V$ be a unitary $K$-module where $K$ is an arbitrary compact Lie group, and let $G$, $Z$ and $N$ be as usual. Then
$Z$ is Zariski dense in $N$ if  and only if $V$ is $1$-large as $G$-module.
\end{lemma}

Let $v\in Z$, $L = K_v$, and $E = T_v(K v)$. The \emph{(real) symplectic slice at $v$} for the $K$-action on $Z$ is given by $E^\omega/E$ where
$E^\omega$ denotes the orthogonal complement with respect to the symplectic form $\omega$ induced by the hermitian metric on $V$.
The action of $L$ on the real symplectic slice is Hamiltonian with respect to the symplectic structure inherited from $V$, and the moment map of the
$L$-action is the restriction of $\rho$. See \cite[Section 6]{HerbigSchwarz} for more details. Note that the complexification of the real symplectic
slice is the (complex) symplectic slice for the $G$-action at $(v_+, v_-)\in N$.
Let $Z_v$ denote the real shell of the real symplectic slice at $v$.
We say that the $K$-module $V$ and real shell $Z$ are \emph{minimal} if $V$ has no slice representations of type $O$, or equivalently, if the complex shell $N$ is minimal.

By a \emph{regular isomorphism} or \emph{regular diffeomorphism} of real symplectic quotients, we mean a homeomorphism that induces an isomorphism of the algebras of real regular functions. A \emph{regular symplectomorphism} is a regular isomorphism such that the induced algebra isomorphism is Poisson. A regular isomorphism or symplectomorphism is \emph{graded} if the algebra isomorphism preserves the grading. See \cite[Section 2]{HerbigSchwarzSeaton} or \cite[Section 4.2]{FarHerSea}.

We have the following real analog of Theorem \ref{mainthm:6}.

 \begin{theorem}
 \label{mainthm:6.real}
Let $K$ be a real quasi-toral compact Lie group and $V$ a  faithful unitary $K$-module with real shell $Z$. We assume that $V$ is stable as $G=K_\C$-module. There is a complex linear subspace $V'\subset V$ with the following properties.
\begin{enumerate}
\item Let  $K'\subset K$ be the stabilizer of $V'$ with complexification $G'$. Then $V'$ is a stable
 faithful $G'$-module.
\item There is a $K'$-equivariant  inclusion $Z'\to Z$ inducing a
    graded regular symplectomorphism $Z'/K'\simeq Z/K$ where $Z'$ is the real shell of $V'$.
\item  $Z'$ is minimal.
\end{enumerate}
\end{theorem}

  \begin{proof}
  Theorem \ref{thm:main3} yields a
  faithful stable (and hence $1$-large) minimal $G'$-module $V'$ and $G'$-equivariant inclusion $N' = N_{V'}\to N$ inducing an isomorphism $N'\git G'\tosim N\git G$.
  Consider  $V$ as a $K$-module and $V'$ as a
  $(K'=G'\cap K)$-module. Then  $V\otimes_\R\C\simeq V\oplus V^*$ as $K$-module and similarly for $V'$ and $K'$. We have a real linear embedding $V'\otimes_\R\C\to V\otimes_\R\C$
 and on the real points we have the embedding of $V'_\R\to V_\R$.
 Then $N\cap V_\R=Z$ and $N'\cap V'_\R=Z'$.
 The real $G'$-equivariant embedding $N'\to N$ induces a
 $K'$-equivariant
 embedding $Z'\to Z$ and quotient mapping $Z'/K'\to Z/K$.

   By Lemma \ref{lem:1-large}, the inclusions $Z\to V$ and $Z'\to V'$ induce     homeomorphisms $Z/K\tosim V\git G$ and $Z'/K'\tosim V'\git G'$.  It now follows from Corollary \ref{cor:V'/G'.to.V/G} that  $Z'/K'\to Z/K$ is a homeomorphism.
 Since
 $Z'\to Z$ is induced by a linear map, the corresponding isomorphism of $\R[Z']^{K'}$ and $\R[Z]^K$ is graded and $Z'/K'\to Z/K$ is a graded regular symplectomorphism.
 \end{proof}

We have the following real version of Corollary \ref{maincor:1}.
\begin{corollary}
Let $V$ be a unitary $K$-module where $K$ is
real
quasi-toral and we suppose that $(V,G=K_\C)$ is $1$-large and that $Z$ is minimal. Then $Z/K$ is an orbifold if and only if $K$ is finite.
\end{corollary}

 \begin{remark}
In  \cite[Theorems 1.1, 1.3, 1.5]{HerbigSchwarzSeaton} we give criteria for $Z/K$ \emph{not} to be an orbifold.  Here the image of $K$ in $\GL(V)$ is positive dimensional and we may assume that $K\to\GL(V)$ is faithful.
We now assume that  $K$ is
real quasi-toral so that $G=K_\C$ is quasi-toral.  We show that the associated $N=N_V$ is minimal, recovering the  cited theorems of \cite{HerbigSchwarzSeaton}
in the real quasi-toral case.

 In Theorem 1.5, $G$ is $\C^\times$ and $(V\oplus V^*,G)$ is not of type $O$, hence no symplectic slice representation of $N$ is either and $N$ is minimal.  In Theorems 1.1 and 1.3, $V$ is $2$-principal and stable, hence
 $2$-large \cite[Theorem 3.2]{HerbigSchwarz} and \cite[Proposition 10.1]{GWSlifting}.
 Then
 $\dim V - \dim V^G-\dim G \geq 2$
 \cite[Lemma 2.2]{HerbigSchwarzSeaton2}, and the same holds for each of the slice representations of $V$
 \cite[Remark 9.6(2)]{GWSlifting}.
 By Lemma \ref{lem:H-type-O}, $N$ has no slice representations of type O,
hence   $N$ is  minimal.
 \end{remark}

We have the following real analog of Theorem \ref{mainthm:3}.

\begin{theorem}\label{thm:isomorphism.real}
For $i=1$, $2$, suppose that $K_i$ is a
real
quasi-toral compact Lie group with complexification $(K_i)_\C = G_i$ and $V_i$ is a faithful unitary $K_i$-module which is stable as a $G_i$-module. Let $Z_i$ denote the real shell and $N_i$
 the complex shell, and assume that the $N_i$ are minimal.
The following are equivalent.
\begin{enumerate}
\item There is a regular isomorphism $\phi\colon Z_1/K_1\to Z_2/K_2$.
\item There is a real  isomorphism  $\Phi\colon N_1\git G_1\to N_2\git G_2$.
\item There is a real linear isomorphism
$$
\Gamma\colon V_1\oplus V_1^*\to V_2\oplus V_2^*
$$
inducing (necessarily real)  isomorphisms of $N_1$ and $N_2$ and of $G_1$ and $G_2$.
\item There is a linear isomorphism
$$
\Gamma'\colon V_1\to V_2
$$
inducing   isomorphisms $Z_1\tosim Z_2$ and $K_1\tosim K_2$.
\end{enumerate}
\end{theorem}

\begin{proof}
The real shell $Z_i$ is Zariski dense in the complex shell $N_i$ for $i=1$, $2$.
Hence $\R[Z_i]\otimes_\R\C=\C[N_i]$ and
$\R[Z_i/K_i]\otimes_\R\C = \R[Z_i]^{K_i}\otimes_\R\C=\C[N_i]^{G_i}$ for $i=1$, $2$. Thus
$(1) \Longrightarrow (2)$ and $(2)\Longrightarrow (3)$ since the $\Gamma$ given by Theorem \ref{thm:isomorphism}   has to induce an isomorphism of $(V_1)_\R$ to $(V_2)_\R$. Clearly  $(4)\Longrightarrow (1)$.
Suppose that (3) holds.
Then $\Gamma$ restricts to a linear isomorphism
of
$(V_1)_\R$ onto $(V_2)_\R$
and hence induces a linear isomorphism $\Gamma'\colon V_1\to V_2$.
Since the real points of the $G_i$ are the $K_i$ for $i=1$, $2$, $\Gamma'$ sends $K_1$ to $K_2$ and $Z_1$ to $Z_2$ giving (4).
\end{proof}

If Theorem \ref{thm:isomorphism.real}(4) holds, then as the induced regular isomorphism $Z_1/K_1\to Z_2/K_2$ of (1) is constructed from a linear
isomorphism $V_1\to V_2$, it preserves the grading and Poisson structures of $\R[Z_1/K_1]$ and $\R[Z_2/K_2]$. Hence we have the following.

\begin{corollary}\label{cor:graded.regular}
With $K_i$, $V_i$, etc. as in Theorem \ref{thm:isomorphism.real}, if there is a regular isomorphism $Z_1/K_1\to Z_2/K_2$, then there is
a graded regular symplectomorphism $Z_1/K_1\to Z_2/K_2$.
\end{corollary}

Unlike the complex case addressed in Theorems \ref{mainthm:6} and \ref{mainthm:3}, the assumption of stability is crucial in Theorems \ref{mainthm:6.real} and \ref{thm:isomorphism.real}.
In the real case, changing the Lagrangian submodule can significantly change the real shell $Z$ and real symplectic quotient $Z/K$,
see \cite[Example 2.10]{HerbigSchwarzSeaton3}.

\begin{example}
(\cite[Proposition 1]{HerbigLawlerSeaton20})
Let $K= \Sp^1$ with complexification $G=\C^\times$. Let  $V_1$ and $V_2$ be $K$-modules with weights   $(-2,3,6)$ and $(-3, 2, 6)$, respectively.  Let $N_1$, $N_2$, $Z_1$ and $Z_2$ be as above where $K_1=K_2=K$ and $G_1=G_2=G$. Let $Y_1$ (resp.\ $Y_2$) denote the real points of $N_1\git G$ (resp.\ $N_2\git G$). Theorem \ref{thm:isomorphism.real}(4)
fails to hold by \cite[Proposition 1]{HerbigLawlerSeaton20}.   However,  $N_1\git G\simeq N_2\git G$ and
$$
\R[Y_1]=\R[Z_1]^K\simeq\R[Z_2]^K=\R[Y_2].
$$
By  \cite{ProcesiSchwarz}, $Z_1/K$ and $Z_2/K$ are semialgebraic subsets of $Y_1$ and $Y_2$, respectively. By Theorem \ref{thm:isomorphism.real}, no   algebraic isomorphisms of $Y_1$ and $Y_2$ can send $Z_1/K$ isomorphically onto $Z_2/K$.
\end{example}

\bibliographystyle{amsalpha}
\bibliography{HSS-torus}

\end{document}